\documentclass{article}

\def\TARGET{1}

\usepackage{xcolor}
\PassOptionsToPackage{table, svgnames, dvipsnames}{xcolor}

\usepackage[final]{neurips_2024}

\usepackage{graphicx}

\usepackage{amssymb,amsmath,amsfonts,amsthm,commath}
\usepackage[algo2e,ruled]{algorithm2e}
\usepackage{hyperref}
\usepackage{multirow}
\usepackage[capitalise]{cleveref}
\usepackage{enumitem}
\usepackage{wrapfig}
\usepackage{colortbl}

\theoremstyle{plain}
\newtheorem{theorem}{Theorem}

\newtheorem{lemma}{Lemma}
\newtheorem{corollary}{Corollary}
\theoremstyle{definition}

\theoremstyle{remark}
\newtheorem{remark}{Remark}

\renewcommand{\t}{\text}
\newcommand{\op}[1]{\operatorname{#1}}
\newcommand{\C}[1]{{\mathcal{#1}}} \newcommand{\B}[1]{{\mathbb{#1}}} \newcommand{\BF}[1]{{\mathbf{#1}}}  \newcommand{\F}[1]{{\mathfrak{#1}}}

\newcommand{\bra}{\langle}
\newcommand{\ket}{\rangle}

\newcommand{\x}{\BF{x}}
\newcommand{\y}{\BF{y}}
\newcommand{\z}{\BF{z}}
\newcommand{\vv}{\BF{v}}
\newcommand{\uu}{\BF{u}}
\newcommand{\oo}{\BF{o}}
\newcommand{\g}{\BF{g}}

\renewcommand{\cite}[1]{\citep{#1}}

\DeclareMathOperator*{\argmin}{argmin}

\makeatletter
\def\blfootnote{\xdef\@thefnmark{}\@footnotetext}
\makeatother

\title{
From Linear to Linearizable Optimization: 
A Novel Framework with Applications to Stationary and Non-stationary DR-submodular Optimization
}
\date{}
\author{
\begin{tabular}{c}
Mohammad Pedramfar \\
{\normalfont McGill University and Mila
\thanks{Work done while at Purdue University}
} \\
\texttt{mohammad.pedramfar@mila.quebec}
\end{tabular}
\And 
\begin{tabular}{c}
Vaneet Aggarwal \\
{\normalfont Purdue University} \\
\texttt{vaneet@purdue.edu}
\end{tabular}
}

\begin{document}

\maketitle

\blfootnote{Compared with the NeurIPS 2024 version, this arXiv revision refines the zeroth-order and bandit-feedback wrappers to simplify their application, with corresponding modifications to the algorithms and their analysis.}

\begin{abstract}
This paper introduces the notion of upper-linearizable/quadratizable functions, a class that extends concavity and DR-submodularity in various settings, including monotone and non-monotone cases. A general meta-algorithm is devised to convert algorithms for linear/quadratic maximization into ones that optimize upper-linearizable/quadratizable functions, offering a unified approach to tackling concave and DR-submodular optimization problems. The paper extends these results to multiple feedback settings, facilitating conversions between semi-bandit/first-order feedback and bandit/zeroth-order feedback, as well as between first/zeroth-order feedback and semi-bandit/bandit feedback. Leveraging this framework, new algorithms are derived using existing results as base algorithms for convex optimization, improving upon state-of-the-art results in various cases. Dynamic and adaptive regret guarantees are obtained for DR-submodular maximization, marking the first algorithms to achieve such guarantees in these settings. Notably, the paper achieves these advancements with fewer assumptions compared to existing state-of-the-art results, underscoring its broad applicability and theoretical contributions to non-convex optimization.
\end{abstract}

\if\TARGET0
    \bibliographystyle{plain}
\else
    \bibliographystyle{apalike}
\fi

\vspace{-.1in}
\section{Introduction}
\vspace{-.1in}


{\bf Overview: } The prominence of optimizing continuous adversarial $\gamma$-weakly up-concave functions {(with DR-submodular and concave functions as special cases)} has surged in recent years, marking a crucial subset within the realm of non-convex optimization challenges, particularly in the forefront of machine learning and statistics. 
This problem has numerous real-world applications, such as revenue maximization, mean-field inference, recommendation systems \cite{bian2019optimal,hassani17_gradien_method_submod_maxim,mitra2021submodular,djolonga2014map,ito2016large,gu2023profit,li2023experimental}. 
This problem is modeled as a repeated game between an optimizer and an adversary. 
In each round, the optimizer selects an action, and the adversary chooses a $\gamma$-weakly up-concave reward function. 
Depending on the scenario, the optimizer can then query this reward function either at any arbitrary point within the domain (called full information feedback) or specifically at the chosen action (called semi-bandit/bandit feedback), where the feedback can be noisy/deterministic.
The performance metric of the algorithm is measured with multiple regret notions - static adversarial regret, dynamic regret, and adaptive regret. 
The algorithms for the problem are separated into the ones that use a projection operator to project the point to the closest point in the domain, and the projection-free methods that replace the projection with an alternative such as Linear Optimization Oracles (LOO) or Separation Oracles (SO).
This interactive framework introduces a range of significant challenges, influenced by the characteristics of the up-concave function (monotone/non-monotone), the constraints imposed, the nature of the queries, projection-free/projection-based algorithms, and the different regret definitions.

\begin{table}[t] 
\vspace{-.3in}
\small
\caption{Online up-concave maximization} 
\vspace{-.1in}
\label{tbl:adv}
\begin{center}
\resizebox{.92\textwidth}{!}{
\begin{tabular}{ | c | c | c | c | c | c | c | c | c | }
\hline
$F$ & Set & \multicolumn{3}{|c|}{Feedback} & Reference & Appx. & \# of queries & $\log_T(\alpha\t{-regret})$ \\
\hline
\multirow{20}{*}{\rotatebox{90}{Monotone}}
& \multirow{12}*{\rotatebox{90}{$0 \in \C{K}$}}
  & \multirow{5}*{$\nabla F$}
    & \multirow{3}*{Full Information} 
      & \multirow{3}*{stoch.}
        & \cite{zhang22_stoch_contin_submod_maxim} $\ddagger$ {\color{blue}(*)}
          & $1-e^{-\gamma}$ & $1$ & $1/2$ \\
& & & & & \cite{pedramfar23_unified_projec_free_algor_adver}
          & $1 - e^{-1}$ & $T^{\theta} (\theta \in [0, 1/2])$ & $2/3 - \theta/3$ \\
& & & & & {\color{blue}Corollary~\ref{cor:w-ada-reg}-b}
          &  $1 - e^{-\gamma}$ &  $1$ & $1/2$ \\
  \cline{4-9}
& & & \multirow{2}*{Semi-bandit}
      & \multirow{2}*{stoch.}
        & \cite{pedramfar23_unified_projec_free_algor_adver}
          & $1 - e^{-1}$ & - & $3/4$ \\
& & & & & {\color{blue}Corollary~\ref{cor:w-ada-reg}-b}
          & $1 - e^{-\gamma}$ & - & $2/3$ \\
  \cline{3-9}
& & \multirow{7}*{$F$}
    & \multirow{3}*{Full Information}
      & \multirow{1}*{det.}
        & {\color{blue}Corollary~\ref{cor:w-ada-reg}-b}
          & $1-e^{-\gamma}$ & $2$ & $1/2$ \\
    \cline{5-9}
& & & & \multirow{2}*{stoch.}
        & \cite{pedramfar23_unified_projec_free_algor_adver}
          & $1-e^{-1}$ & $T^{\theta} (\theta \in [0, 1/4])$ & $4/5 - \theta/5$ \\
& & & & & {\color{blue}Corollary~\ref{cor:w-ada-reg}-b}
          & $1 - e^{-\gamma}$ & $1$ & $3/4$ \\
  \cline{4-9}
& & & \multirow{4}*{Bandit}
      & \multirow{2}*{det.}
        & \cite{wan23_bandit_multi_dr_submod_maxim} $\ddagger\ddagger$
          & $1-e^{-1}$ & - & $3/4$ \\
& & & & & \cite{zhang24_boost_gradien_ascen_contin_dr_maxim} $\ddagger${\color{blue}(*)}
          & $1-e^{-\gamma}$ & - & $4/5$ \\
  \cline{5-9}
& & & & \multirow{2}*{stoch.}
        & \cite{pedramfar23_unified_projec_free_algor_adver}
          & $1-e^{-1}$ & - & $5/6$ \\
& & & & & {\color{blue}Corollary~\ref{cor:w-ada-reg}-b}
          & $1 - e^{-\gamma}$ & - & $4/5$ \\
  \cline{2-9}
& \multirow{8}*{\rotatebox{90}{general}} 
  & \multirow{4}*{$\nabla F$}
    & Full Information
      & stoch.
        & \cite{pedramfar23_unified_projec_free_algor_adver}
          & $1/2$ & $T^{\theta} (\theta \in [0, 1/2])$ & $2/3 - \theta/3$ \\
  \cline{4-9}
& & & \multirow{3}*{Semi-bandit}
      & \multirow{3}*{stoch.}
        & \cite{chen18_onlin_contin_submod_maxim}$\ddagger${\color{blue}(*)}
          & $\gamma^2/(1 + \gamma^2)$ & - & $1/2$ \\
& & & & & \cite{pedramfar23_unified_projec_free_algor_adver}
          & $1/2$ & - & $3/4$ \\
& & & & & {\color{blue}Corollary~\ref{cor:w-ada-reg}-a}
          & $\gamma^2/(1 + c\gamma^2)$ & - & $1/2$ \\
  \cline{3-9}
& & \multirow{4}*{$F$}
    & \multirow{2}*{Full Information}
      & det.
        & {\color{blue}Corollary~\ref{cor:w-ada-reg}-a}
          & $\gamma^2/(1 + c\gamma^2)$ & $2$ & $1/2$ \\
      \cline{5-9}
& & & & stoch.
        & \cite{pedramfar23_unified_projec_free_algor_adver}
          & $1/2$ & $T^{\theta} (\theta \in [0, 1/4])$ & $4/5 - \theta/5$ \\
  \cline{4-9}
& & & \multirow{2}*{Bandit}
      & \multirow{2}*{stoch.}
        & \cite{pedramfar23_unified_projec_free_algor_adver}
          & $1/2$ & - & $5/6$ \\
& & & & & {\color{blue}Corollary~\ref{cor:w-ada-reg}-a}
          & $\gamma^2/(1 + c\gamma^2)$ & - & $3/4$ \\
\hline
\multirow{11}*{\rotatebox{90}{Non-Monotone}}
& \multirow{11}*{\rotatebox{90}{general}}
  & \multirow{5}*{$\nabla F$}
    & \multirow{3}*{Full Information}
      & \multirow{3}*{stoch.}
        & \cite{pedramfar23_unified_projec_free_algor_adver}
          & $(1-h)/4$ & $T^{\theta} (\theta \in [0, 1/2])$ & $2/3 - \theta/3$ \\
& & & & & \cite{zhang24_boost_gradien_ascen_contin_dr_maxim} $\ddagger${\color{blue}(*)}
          & $(1-h)/4$ & $1$ & $1/2$ \\
& & & & & {\color{blue}Corollary~\ref{cor:w-ada-reg}-c}
          & $(1-h)/4$ & $1$ & $1/2$ \\
  \cline{4-9}
& & & \multirow{2}*{Semi-bandit}
      & \multirow{2}*{stoch.}
        & \cite{pedramfar23_unified_projec_free_algor_adver}
          & $(1-h)/4$  & - & $3/4$ \\
& & & & & {\color{blue}Corollary~\ref{cor:w-ada-reg}-c}
          & $(1-h)/4$  & - & $2/3$ \\
  \cline{3-9}
& & \multirow{6}*{$F$}
    & \multirow{3}*{Full Information}
      & \multirow{1}*{det.}
        & {\color{blue}Corollary~\ref{cor:w-ada-reg}-c}
          & $(1-h)/4$ & $2$ & $1/2$ \\
        \cline{5-9}
& & & & \multirow{2}*{stoch.}
        & \cite{pedramfar23_unified_projec_free_algor_adver}
          & $(1-h)/4$ & $T^{\theta} (\theta \in [0, 1/4])$ & $4/5 - \theta/5$ \\
& & & & & {\color{blue}Corollary~\ref{cor:w-ada-reg}-c}
          & $(1-h)/4$  & $1$ & $3/4$ \\
  \cline{4-9}
& & & \multirow{3}*{Bandit}
      & \multirow{1}*{det.}
        & \cite{zhang24_boost_gradien_ascen_contin_dr_maxim} $\ddagger${\color{blue}(*)}
          & $(1-h)/4$ & - & $4/5$ \\
    \cline{5-9}
& & & & \multirow{2}*{stoch.}
        & \cite{pedramfar23_unified_projec_free_algor_adver}
          & $(1-h)/4$ & - & $5/6$ \\
& & & & & {\color{blue}Corollary~\ref{cor:w-ada-reg}-c}
          & $(1-h)/4$ & - & $4/5$ \\
\hline
\end{tabular}
}
\end{center}
{~\\\small 
This table compares different static regret results for the online up-concave maximization. 
The logarithmic terms in regret are ignored.
Here $h := \min_{\z \in \C{K}} \|\z\|_\infty$.
Our algorithm is projection-free and use a separation oracle. 
The rows marked with $\ddagger$ use gradient ascent, requiring potentially computationally expensive projections. 
The rows marked with ${\color{blue}(*)}$ denote results that could be considered special case of our framework. In particular, we obtain those results if we use Online Gradient Ascent instead of $\mathtt{SO\textrm{-}OGA}$ as the base algorithm in Corollary~\ref{cor:w-ada-reg}.
Note that the result of \cite{wan23_bandit_multi_dr_submod_maxim}, marked by $\ddagger \ddagger$, uses a convex optimization subroutine in each iteration, which could potentially be more expensive than projection and therefore not considered a projection-free result.
\textbf{It is also the only existing result, in all the tables, that outperforms ours.} We note that stochastic results can be used in deterministic, and bandit/semi-bandit in full information and thus cases where our result is not added, our result improves state-of-the-art projection free result because of the result in less information setup. 

All results assume that functions are Lipschitz.
Our first-order results for monotone functions over general convex sets allow up-concavity with respect to a specified up-super-gradient field and do not require differentiability. Our zeroth-order results assume differentiable Lipschitz extensions to the cube that are up-concave with respect to their own gradients, as specified in the application corollaries; the other first-order results also assume differentiability.
All previous results assume that functions are DR-submodular, while we only require up-concavity.
Results of~\cite{pedramfar23_unified_projec_free_algor_adver} and~\cite{wan23_bandit_multi_dr_submod_maxim} also assume functions are smooth, i.e., their gradients are Lipschitz.
}
\vspace{-.2in}
\end{table}

In this paper, we present a comprehensive approach to solving adversarial up-concave optimization problems, encompassing different feedback types (including bandit, semi-bandit and full-information feedback), characteristics of the up-concave function and constraint region, projection-free/projection-based algorithms, and regret definitions.
While the problem has been studied in many special cases, the main contribution of this work is a framework that is based on a novel notion of the function class being upper-linearizable (or upper-quadratizable).
We design a meta-algorithm that converts certain algorithms designed for online linear maximization to algorithms capable of handling upper-linearizable function classes.
This allows us to reduce the problem of up-concave maximization in three different settings to online linear maximization and obtain corresponding regret bounds.
In particular, our results include monotone $\gamma$-weakly up-concave functions over general convex set, monotone $\gamma$-weakly up-concave functions over convex sets containing the origin and non-monotone up-concave functions. While the above result is for first order feedback, we then derive multiple results that increase the applicability of the above results.
We extend the applicability of $\mathtt{FOTZO}$ and $\mathtt{STB}$ algorithm introduced in~\cite{pedramfar24_unified_framew_analy_meta_onlin_convex_optim} to our setting which allows us to convert algorithms for first-order/semi-bandit feedback into algorithms for zeroth-order/bandit feedback.
We also design a meta-algorithm that allows us to convert algorithms that require full-information feedback into algorithms that only require semi-bandit/bandit feedback.

We demonstrate the usefulness of results through two applications as described in the following.  In the first application, we use the $\mathtt{SO\textrm{-}OGD}$ Algorithm in~\cite{garber22_new_projec_algor_onlin_convex} as the base algorithm for online linear optimization, which is a projection-free algorithm.
Using this, we first obtain the adaptive regret (and therefore also static regret) guarantees for the three setups of DR-submodular (or more generally, up-concave) optimization with semi-bandit feedback/first order feedback in the respective cases.
Then, the meta-algorithms for conversion of first-order/semi-bandit to zeroth-order/bandit are used to get result with zeroth-order/bandit feedback.
In the cases where the algorithms are full-information and not (semi-)bandit, we use another meta-algorithm to obtain algorithms in (semi-)bandit feedback setting. In the next application, we use the ``Improved Ader" algorithm of~\cite{zhang18_adapt_onlin_learn_dynam_envir} which is a projection based algorithm providing dynamic regret guarantees for the convex optimization.
Afterwards, the same approach as above are used to obtain the results in the three scenarios of up-concave optimization with first-order feedback.

{\bf Technical Novelty: } The main technical novelties in this work are as follows.

\begin{enumerate}[left=10pt]


\item We proposes a novel notion of linearizable/quadratizable functions and extend the meta-algorithm framework of~\cite{pedramfar24_unified_framew_analy_meta_onlin_convex_optim} from convex functions to linearizable/quadratizable functions.
This allows us to relates a large class of algorithms and regret guarantees for optimization of linear/quadratic functions to that for linearizable/quadratizable functions.

\item We show that the class of quadratizable function optimization is general, and includes not only concave, but up-concave optimization in several cases.  
For some of the cases, this proof uses a generalization of the idea of boosting (\cite{zhang22_stoch_contin_submod_maxim,zhang24_boost_gradien_ascen_contin_dr_maxim}) which was proposed for DR-submodular maximization, as mentioned in Corollaries~\ref{cor:dr_mono_zero:boosting} and~\ref{cor:dr_nonmono:boosting}.


\item We design a new meta-algorithm, namely $\mathtt{SFTT}$, that captures the idea of random permutations (sometimes referred to as blocking) as used in several papers such as~\cite{zhang19_onlin_contin_submod_maxim,zhang23_onlin_learn_non_submod_maxim,pedramfar23_unified_projec_free_algor_adver}.
While previous works used this idea in specific settings, our meta-algorithm is applicable in general settings.

\item We note the generality of the above results in this paper.
Our results are general in the following three aspects:

\noindent a) In this work, we improve results for projection-free static regret guarantees for DR-submodular optimization in all considered cases and obtain the first results for dynamic and adaptive regret.
Moreover, these guarantees follow from existing algorithms for the linear  optimization, using only the statement of the regret bounds and simple properties of the algorithms.

\noindent b)  We consider 3 classes of DR-submodular functions in this work.
However, to extend these results to another function class, all one needs to do is to (i) prove that the function class is quadratizable; and  (ii) provide an unbiased estimator of $\F{g}$ (as described in Equation~\ref{eq:quadratizable}).

\noindent c)  We consider 2 different feedback types in offline setting (first/zero order) and 4 types of feedback in the online setting (first/zero order and full-information/trivial query).
Converting results between different cases is obtained through meta-algorithms and guarantees for the meta-algorithms which only relies on high level properties of the base algorithms (See Theorems~\ref{thm:first-order-to-zero-order},~\ref{thm:stoch-fi-to-sb},~\ref{thm:first-order-to-det-zero-order} and~\ref{thm:online-to-offline})
\end{enumerate}

{\bf Key contributions: } The key contributions in this work are summarized as follows. 

\begin{enumerate}[left=10pt]

\item 
We formulate the notion of \textit{upper-quadratizable/upper-linearizeble} functions, which is a class that generalizes the notion of strong-concavity/concavity and also DR-submodularity in several settings.
In particular, we demonstrate the the following function classes are upper-quadratizable:
(i) monotone $\gamma$-weakly $\mu$-strongly DR-submodular functions with curvature $c$ over general convex sets, 
(ii) monotone $\gamma$-weakly DR-submodular functions over convex sets containing the origin, and 
(iii) non-monotone DR-submodular optimization over general convex sets.

\item 
We provide a general meta-algorithm that converts algorithms for linear/quadratic maximization to algorithms that maximize upper-quadratizable functions.
This results is a unified approach to maximize both concave functions and DR-submodular functions in several settings.

\item 
While the above provides results for semi-bandit feedback (for monotone  DR-submodular optimization over general convex sets) and first-order feedback (for monotone  DR-submodular optimization over  convex sets containing the origin, and non-monotone  DR-submodular optimization over general convex sets), the results could be extended to more general feedback settings.
Four meta algorithms are provided that relate semi-bandit/first-order feedback to bandit/zeroth order feedback; that relate; first order to deterministic zeroth order; and that relate first/zeroth order feedback to semi-bandit/bandit feedback.
Together they allow us to obtain results in 5 feedback settings (first/zeroth order full-information and semi-bandit/bandit; and deterministic zeroth order).
We also discuss a meta-algorithm to convert online results to offline guarantees.

\item 
The above framework is applied using different algorithms as the base algorithms for linear optimization.
$\mathtt{SO\textrm{-}OGD}$~\cite{garber22_new_projec_algor_onlin_convex} is a projection-free algorithm using separation oracles that provides adaptive regret guarantees for online convex optimization. 
We use our framework to cover the 18 cases in Table~\ref{tbl:adv}.
We improve the regret guarantees for the previous SOTA projection-free algorithms in all the cases.
If we also allow comparisons with the algorithms that are not projection-free, we still improve the SOTA results in 12 cases and match the SOTA in 5 cases.

\item 
Using our framework, we convert online results using $\mathtt{SO\textrm{-}OGD}$ to offline results covering the nine feedback/function-class cases in Table~\ref{tbl:offline}.
We improve the regret guarantees for the previous SOTA projection-free algorithms in all the cases, except for deterministic first order feedback where existing results are already SOTA.
If we also allow comparisons with the algorithms that are not projection-free, we still improve the SOTA results in 6 cases and match the SOTA in the remaining 3 cases.

\item
We use our framework to convert the adaptive regret guarantees of $\mathtt{SO\textrm{-}OGD}$ to obtain projection-free algorithms with adaptive regret bounds that cover all cases in Table~\ref{tbl:non-stationary}.
Our results are first algorithms with adaptive regret guarantee for online DR-submodular maximization.

\item 
``Improved Ader"~\cite{zhang18_adapt_onlin_learn_dynam_envir} is an algorithm providing dynamic regret guarantees for online convex optimization. 
We use our framework to obtain dynamic regret guarantees for the nine feedback/function-class cases shown in Table~\ref{tbl:non-stationary}.
Our results are first algorithms with dynamic regret guarantee for online DR-submodular maximization.

\item
For monotone $\gamma$-weakly functions with bounded curvature over general convex sets, we improve the approximation ratio (See Lemma~\ref{lem:dr_mono_general:quad}).

\item
As mentioned in the descriptions of the tables, in all cases considered, whenever there is another existing result, we obtain our results using fewer assumptions than the existing SOTA.
\end{enumerate}

\section{Problem Setup and Definitions}\label{main_def}
Throughout the paper, the feasible set $\C{K}\subseteq\B{R}^d$ is assumed to be compact and convex and to contain at least two distinct points. In particular, $\op{diam}(\C{K})>0$ and $\op{dim}(\op{aff}(\C{K}))\geq1$.

For a set $\C{D} \subseteq \B{R}^d$, we define its \textit{affine hull}  $\op{aff}(\C{D})$ to be the set of $\alpha \x + (1 - \alpha) \y$ for all $\x, \y$ in $\C{K}$ and $\alpha \in \B{R}$.
The \textit{relative interior} of $\C{D}$ is defined as
$\op{relint}(\C{D}) := \{ \x \in \C{D} \mid \exists r > 0, \B{B}_r(\x) \cap \op{aff}(\C{D}) \subseteq \C{D} \}.$  For any $\uu \in \C{K}^T$, we define the path length 
$P_T(\uu) := \sum_{i = 1}^{T-1} \| \uu_i - \uu_{i+1} \|.$ 
Given $\mu \geq 0$ and $0 < \gamma \leq 1$, we say a differentiable function $f : \C{K} \to \B{R}$ is \textit{$\mu$-strongly $\gamma$-weakly up-concave} if it is $\mu$-strongly $\gamma$-weakly concave along positive directions.
Specifically if, for all $\x \leq \y$ in $\C{K}$, we have 
\[\gamma \left( \bra \nabla f(\y), \y - \x \ket + \frac{\mu}{2} \| \y - \x \|^2 \right)
\leq f(\y) - f(\x)
\leq \frac{1}{\gamma} \left( \bra \nabla f(\x), \y - \x \ket - \frac{\mu}{2} \| \y - \x \|^2 \right). \]

\begin{table}[t]
\vspace{-.4in}
\small
\caption{\small Offline up-concave maximization} 
\label{tbl:offline}
\vspace{-.05in}
\begin{center}
\resizebox{.6\textwidth}{!}{
\begin{tabular}{ | c | c | c | c | c | c | c | }
\hline
$F$ & Set & \multicolumn{2}{|c|}{Feedback} & Reference & Appx. & Complexity \\
\hline
\multirow{10}{*}{\rotatebox{90}{Monotone}}
& \multirow{6}*{\rotatebox{90}{$0 \in \C{K}$}}
  & \multirow{4}*{$\nabla F$}
    & \multirow{4}*{stoch.}
      & \cite{mokhtari20_stoch_condit_gradien_method}
        & $1-e^{-\gamma}$ & $O(1/\epsilon^3)$ \\
& & & & \cite{hassani20_stoch_condit_gradien}
         & $1-e^{-\gamma}$ & $O(1/\epsilon^2)$ \\
& & & & \cite{zhang22_stoch_contin_submod_maxim} $\ddagger$
          & $1-e^{-\gamma}$ & $O(1/\epsilon^2)$ \\
& & & & {\color{blue}Corollary~\ref{cor:w-ada-reg}-b}
          & $1-e^{-\gamma}$ & $O(1/\epsilon^2)$ \\
  \cline{3-7}
& & \multirow{2}*{$F$}
    & det.
      & {\color{blue}Corollary~\ref{cor:w-ada-reg}-b}
        & $1-e^{-\gamma}$ & $O(1/\epsilon^2)$ \\
    \cline{4-7}
& & & stoch.
      & {\color{blue}Corollary~\ref{cor:w-ada-reg}-b}
        & $1-e^{-\gamma}$ & $O(1/\epsilon^4)$ \\
\cline{2-7}
& \multirow{4}*{\rotatebox{90}{general}}
  & \multirow{2}*{$\nabla F$}
    & \multirow{2}*{stoch.}
      & \cite{hassani17_gradien_method_submod_maxim}$\ddagger$
        & $\gamma^2/(1 + \gamma^2)$ & $O(1/\epsilon^2)$ \\
& & & & {\color{blue}Corollary~\ref{cor:w-ada-reg}-a}
          & $\gamma^2/(1 + c\gamma^2)$ & $O(1/\epsilon^2)$ \\
  \cline{3-7}
& & \multirow{2}*{$F$}
    & det.
      & {\color{blue}Corollary~\ref{cor:w-ada-reg}-a}
        & $\gamma^2/(1 + c\gamma^2)$ & $O(1/\epsilon^2)$ \\
    \cline{4-7}
& & & stoch.
      & {\color{blue}Corollary~\ref{cor:w-ada-reg}-a}
        & $\gamma^2/(1 + c\gamma^2)$ & $O(1/\epsilon^4)$ \\
\hline
\multirow{4}*{\rotatebox{90}{Non-Monotone}}
& \multirow{4}*{\rotatebox{90}{general}}
  & \multirow{2}*{$\nabla F$}
    & \multirow{2}*{stoch.}
      & \cite{zhang24_boost_gradien_ascen_contin_dr_maxim} $\ddagger$
        & $(1-h)/4$ & $O(1/\epsilon^2)$ \\
& & & & {\color{blue}Corollary~\ref{cor:w-ada-reg}-c}
        & $(1-h)/4$ & $O(1/\epsilon^2)$ \\
\cline{3-7}
& & \multirow{2}*{$F$}
    & det.
      & {\color{blue}Corollary~\ref{cor:w-ada-reg}-c}
        & $(1-h)/4$ & $O(1/\epsilon^2)$ \\
    \cline{4-7}
& & & stoch.
      & {\color{blue}Corollary~\ref{cor:w-ada-reg}-c}
        & $(1-h)/4$ & $O(1/\epsilon^4)$ \\
\hline
\end{tabular}
}
\end{center}
{~\\ \small
This table compares the different results for the number of oracle calls (complexity) \textit{within the constraint set} for up-concave maximization. 
Here $h := \min_{\z \in \C{K}} \|\z\|_\infty$.

$\ddagger$ \cite{hassani17_gradien_method_submod_maxim},~\cite{zhang22_stoch_contin_submod_maxim} and~\cite{zhang24_boost_gradien_ascen_contin_dr_maxim} use gradient ascent, requiring potentially computationally expensive projections. 

All previous results assume that functions are differentiable, DR-submodular, Lipschitz and smooth (i.e., their gradients are Lipschitz).
Result of~\cite{hassani20_stoch_condit_gradien} also requires the function Hessians to be Lipschitz. It also requires the density of the stochastic oracle to be known and the log of density to be 4 times differentiable with bounded 4th derivatives.
Our results require non-negative Lipschitz functions. The first-order results for monotone functions over general convex sets allow up-concavity with respect to a specified up-super-gradient field without requiring differentiability. Our zeroth-order results require differentiable Lipschitz extensions to the cube that are up-concave with respect to their own gradients; the other first-order results also require differentiability.
}
\vspace{-.2in}
\end{table}

We say $\tilde{\nabla} f : \C{K} \to \B{R}^d$ is a \textit{$\mu$-strongly $\gamma$-weakly up-super-gradient} of $f$ if for all $\x \leq \y$ in $\C{K}$, the above holds with $\tilde{\nabla}$ instead of ${\nabla}$. 
We say a differentiable function $f$ is $\mu$-strongly $\gamma$-weakly up-concave if $\nabla f$ is a $\mu$-strongly $\gamma$-weakly up-super-gradient for $f$.
More generally, we say $f$ is $\mu$-strongly $\gamma$-weakly up-concave with respect to the vector-valued function $\tilde{\nabla} f : \C{K} \to \B{R}^d$ if $\tilde{\nabla} f$ is a $\mu$-strongly $\gamma$-weakly up-super-gradient for $f$.
Unless a vector field is explicitly specified, up-concavity refers to the differentiable case with respect to the function's own gradient.
When $\gamma = 1$ and the above inequality holds for all $\x, \y \in \C{K}$, we say $f$ is $\mu$-strongly concave. A differentiable function $f : \C{K} \to \B{R}$ is called  \textit{$\gamma$-weakly continuous DR-submodular} if for all $\x \leq \y$, we have $\nabla f(\x) \geq \gamma \nabla f(\y)$.
It follows that any $\gamma$-weakly continuous DR-submodular functions is $\gamma$-weakly up-concave. 
We refer to Appendix~\ref{apdx:setup} for more details.

Given a continuous monotone function $f : \C{K} \to \B{R}$, its curvature is defined as the smallest number $c \in [0, 1]$ such that 
$
f(\y + \z) - f(\y) \geq (1 - c) (f(\x + \z) - f(\x)),
$
 for all $\x, \y \in \C{K}$ and $\z \geq 0$ such that $\x + \z, \y + \z \in \C{K}$. 
We define the curvature of a function class $\BF{F}$ as the supremum of the curvature of functions in $\BF{F}$.

Online optimization problems can be formalized as a repeated game between an agent and an adversary.
The game lasts for $T$ rounds on a convex domain $\C{K}$ where $T$ and $\C{K}$ are known to both players.
In $t$-th round, the agent chooses an action $\x_t$ from an action set $\C{K} \subseteq \B{R}^d$, then the adversary chooses a loss function $f_t \in \BF{F}$ and a query oracle for the function $f_t$.
Then, for $1 \leq i \leq k_t$, the agent chooses a points $\y_{t, i}$ and receives the output of the query oracle. The precise definition of agent $(\Omega^\C{A}, \C{A}^{\t{action}}, \C{A}^{\t{query}})$ is given in Appendix \ref{apdx:setup}, with the  query oracle being any of stochastic/deterministic first/zeroth order or semi-bandit/bandit.

An adversary $\op{Adv}$ is a set of realized adversaries $\C{B} = (\C{B}_1, \cdots, \C{B}_T)$, where each $\C{B}_t$ maps  $(\x_1, \cdots, \x_t) \in \C{K}^T$ to $(f_t, \C{Q}_t)$ where $f_t \in \BF{F}$ and $\C{Q}_t$ is a query oracle for $f_t$. 
Adversaries can be oblivious ($\C{B}_t$ are constant and independent of $(\x_1, \cdots, \x_t)$), weakly adaptive ($\C{B}_t$ are independent of $\x_t$), or fully adaptive (no restrictions).
We use $\op{Adv}^\t{f}_i(\BF{F})$ to denote the set of all possible realized adversaries with deterministic $i$-th order oracles.
If the oracle is instead stochastic and bounded by $B$, we use $\op{Adv}^\t{f}_i(\BF{F}, B)$ to denote such an adversary.
Finally, we use $\op{Adv}^\t{o}_i(\BF{F})$ and $\op{Adv}^\t{o}_i(\BF{F}, B)$ to denote all oblivious realized adversaries with $i$-th order deterministic and stochastic oracles, respectively. In order to handle different notions of regret with the same approach, for an agent $\C{A}$, adversary $\op{Adv}$, compact set $\C{U} \subseteq \C{K}^T$, approximation coefficient $0 < \alpha \leq 1$ and $1 \leq a \leq b \leq T$, we define \textit{regret} as $
	\C{R}_{\alpha, \op{Adv}}^{\C{A}}(\C{U})[a, b] := 
	\sup_{\C{B} \in \op{Adv}} \B{E} \left[ \alpha \max_{\uu = (\uu_1, \cdots, \uu_T) \in \C{U}} \sum_{t = a}^b f_t(\uu_t) - \sum_{t = a}^b f_t(\x_t) \right],$
where the expectation in the definition of the regret is over the randomness of the algorithm and the query oracle.
We use the notation $\C{R}_{\alpha, \C{B}}^{\C{A}}(\C{U})[a, b] := \C{R}_{\alpha, \op{Adv}}^{\C{A}}(\C{U})[a, b]$ when $\op{Adv} = \{\C{B}\}$ is a singleton.
We may drop $\alpha$ when it is equal to 1.
When $\alpha < 1$, we often assume that the functions are non-negative. \textit{Static adversarial regret} or simply \textit{adversarial regret} corresponds to $a = 1$, $b = T$ and $\C{U} = \C{K}_{\star}^T := \{(\x, \cdots, \x) \mid \x \in \C{K}\}$.
When $a = 1$, $b = T$ and $\C{U}$ contains only a single element then it is referred to as the \textit{dynamic regret} \cite{zinkevich03_onlin,zhang18_adapt_onlin_learn_dynam_envir}. 
\textit{Adaptive regret} is defined as
$\max_{1 \leq a \leq b \leq T} \C{R}_{\alpha, \op{Adv}}^{\C{A}}(\C{K}_{\star}^T)[a, b]$ \cite{hazan09_effic}. 
We drop $a$, $b$ and $\C{U}$ when the statement is independent of their value or their value is clear from the context.

\begin{table}[t] 
\small
\caption{Non-stationary up-concave maximization}
\label{tbl:non-stationary}
{\centering
\resizebox{\textwidth}{!}{
\begin{tabular}{ | c | c | c | c | c | c | c | c | c | }
\hline
$F$ & Set & \multicolumn{3}{|c|}{Feedback} & Reference & Appx. & regret type & $\alpha\t{-regret}$ \\
\hline
\multirow{14}{*}{\rotatebox{90}{Monotone}}
& \multirow{8}*{\rotatebox{90}{$0 \in \C{K}$}}
  & \multirow{3}*{$\nabla F$}
    & \multirow{2}*{Full Information}
      & \multirow{2}*{stoch.}
        & {\color{blue}Corollary~\ref{cor:d-reg}}-b
          & $1 - e^{-\gamma}$ & dynamic & $T^{1/2}(1 + P_T)^{1/2}$ \\
& & & & & {\color{blue}Corollary~\ref{cor:w-ada-reg}}-b
          & $1 - e^{-\gamma}$ & adaptive & $T^{1/2}$ \\
  \cline{4-9}
& & & \multirow{1}*{Semi-bandit}
      & \multirow{1}*{stoch.}
        & {\color{blue}Corollary~\ref{cor:w-ada-reg}}-b
          & $1 - e^{-\gamma}$ & adaptive & $T^{2/3}$ \\
  \cline{3-9}
& & \multirow{5}*{$F$}
    & \multirow{4}*{Full Information}
      & \multirow{2}*{det.}
        & {\color{blue}Corollary~\ref{cor:d-reg}}-b
          & $1 - e^{-\gamma}$ & dynamic & $T^{1/2}(1 + P_T)^{1/2}$ \\
& & & & & {\color{blue}Corollary~\ref{cor:w-ada-reg}}-b
            & $1 - e^{-\gamma}$ & adaptive & $T^{1/2}$ \\
      \cline{5-9}
& & & & \multirow{2}*{stoch.}
        & {\color{blue}Corollary~\ref{cor:d-reg}}-b
          & $1 - e^{-\gamma}$ & dynamic & $T^{3/4}(1 + P_T)^{1/2}$ \\
& & & & & {\color{blue}Corollary~\ref{cor:w-ada-reg}}-b
            & $1 - e^{-\gamma}$ & adaptive & $T^{3/4}$ \\
  \cline{4-9}
& & & \multirow{1}*{Bandit}
      & \multirow{1}*{stoch.}
        & {\color{blue}Corollary~\ref{cor:w-ada-reg}}-b
          & $1 - e^{-\gamma}$ & adaptive & $T^{4/5}$ \\
  \cline{2-9}
& \multirow{6}*{\rotatebox{90}{general}} 
  & \multirow{2}*{$\nabla F$}
    & \multirow{2}*{Semi-bandit}
      & \multirow{2}*{stoch.}
        & {\color{blue}Corollary~\ref{cor:d-reg}}-a
          & $\gamma^2/(1 + c\gamma^2)$ & dynamic & $T^{1/2}(1 + P_T)^{1/2}$ \\
& & & & & {\color{blue}Corollary~\ref{cor:w-ada-reg}}-a
          & $\gamma^2/(1 + c\gamma^2)$ & adaptive & $T^{1/2}$ \\
  \cline{3-9}
& & \multirow{4}*{$F$}
    & \multirow{2}*{Full Information}
      & \multirow{2}*{det.}
        & {\color{blue}Corollary~\ref{cor:d-reg}}-a
          & $\gamma^2/(1 + c\gamma^2)$ & dynamic & $T^{1/2}(1 + P_T)^{1/2}$ \\
& & & & & {\color{blue}Corollary~\ref{cor:w-ada-reg}}-a
            & $\gamma^2/(1 + c\gamma^2)$ & adaptive & $T^{1/2}$ \\
    \cline{4-9}
& & & \multirow{2}*{Bandit}
      & \multirow{2}*{stoch.}
        & {\color{blue}Corollary~\ref{cor:d-reg}}-a
          & $\gamma^2/(1 + c\gamma^2)$ & dynamic & $T^{3/4}(1 + P_T)^{1/2}$ \\
& & & & & {\color{blue}Corollary~\ref{cor:w-ada-reg}}-a
        & $\gamma^2/(1 + c\gamma^2)$ & adaptive & $T^{3/4}$ \\
\hline
\multirow{8}*{\rotatebox{90}{Non-Monotone}}
& \multirow{8}*{\rotatebox{90}{general}}
  & \multirow{3}*{$\nabla F$}
    & \multirow{2}*{Full Information} 
      & \multirow{2}*{stoch.}
        & {\color{blue}Corollary~\ref{cor:d-reg}}-c
          & $(1-h)/4$ & dynamic & $T^{1/2}(1 + P_T)^{1/2}$ \\
& & & & & {\color{blue}Corollary~\ref{cor:w-ada-reg}}-c
          & $(1-h)/4$ & adaptive & $T^{1/2}$ \\
  \cline{4-9}
& & & \multirow{1}*{Semi-bandit}
      & \multirow{1}*{stoch.}
        & {\color{blue}Corollary~\ref{cor:w-ada-reg}}-c
          & $(1-h)/4$ & adaptive & $T^{2/3}$ \\
  \cline{3-9}
& & \multirow{5}*{$F$}
    & \multirow{4}*{Full Information}
      & \multirow{2}*{det.}
        & {\color{blue}Corollary~\ref{cor:d-reg}}-c
          & $(1-h)/4$ & dynamic & $T^{1/2}(1 + P_T)^{1/2}$ \\
& & & & & {\color{blue}Corollary~\ref{cor:w-ada-reg}}-c
          & $(1-h)/4$ & adaptive & $T^{1/2}$ \\
    \cline{5-9}
& & & & \multirow{2}*{stoch.}
        & {\color{blue}Corollary~\ref{cor:d-reg}}-c
          & $(1-h)/4$ & dynamic & $T^{3/4}(1 + P_T)^{1/2}$ \\
& & & & & {\color{blue}Corollary~\ref{cor:w-ada-reg}}-c
          & $(1-h)/4$ & adaptive & $T^{3/4}$ \\
  \cline{4-9}
& & & \multirow{1}*{Bandit}
      & \multirow{1}*{stoch.}
        & {\color{blue}Corollary~\ref{cor:w-ada-reg}}-c
          & $(1-h)/4$ & adaptive & $T^{4/5}$ \\
\hline
\end{tabular}
}}
{~\\
\small This table includes different results for non-stationary up-concave maximization, while no prior results exist in this setup to the best of our knowledge.
The results for adaptive regret are projection-free and use a separation oracle while results for dynamic regret use convex projection.
Note that full-information algorithms with deterministic feedback require 2 queries per function while the ones with stochastic feedback only require one, at the cost of higher regret.
}
\vspace{-.2in}
\end{table}

\section{Formulation of Upper-Quadratizable Functions and Regret Relation to that of Quadratic Functions}

Let $\C{K} \subseteq \B{R}^d$ be a convex set, $\BF{F}$ be a function class over $\C{K}$.
We say the function class $\BF{F}$ is \textit{upper-quadratizable} if there are maps $\F{g} : \BF{F} \times \C{K} \to \B{R}^d$ and $h : \C{K} \to \C{K}$ and constants $\mu \geq 0$, $0 < \alpha \leq 1$ and $\beta > 0$ such that
\footnote{
Note that, without any loss in generality, we may replace $(\beta, \mu, \F{g})$ with $(1, \beta\mu, \beta\F{g})$ and therefore assume $\beta = 1$.
However, we keep $\beta$ as a separate variable as it makes some expressions in future sections simpler.
}
\begin{align}\label{eq:quadratizable}
\alpha f(\y) - f(h(\x))
\leq \beta \left( \bra \F{g}(f, \x), \y - \x \ket - \frac{\mu}{2} \| \y - \x \|^2 \right),
\end{align}
As a special case, when $\mu = 0$, we say $\BF{F}$ is \textit{upper-linearizable}.
We use the notation $\BF{F}_{\mu, \BF{g}}$ to denote the class of functions $q(\y) := \bra \F{g}(f, \x), \y - \x \ket - \frac{\mu}{2} \| \y - \x \|^2 : \C{K} \to \B{R},$ for all $f \in \BF{F}$ and $\x \in \C{K}$.
Similarly, for any $B_1 > 0$, we use the notation $\BF{Q}_\mu[B_1]$ to denote the class of functions $
q(\y) := \bra \oo, \y - \x \ket - \frac{\mu}{2} \| \y - \x \|^2 : \C{K} \to \B{R},$
 for all $\x \in \C{K}$ and $\oo \in \B{B}_{B_1}(\BF{0})$.
A similar notion of \textit{lower-quadratizable/linearizable} may be similarly defined for minimization problems such as convex minimization.
\footnote{
Specifically, we say $\BF{F}$ is lower-quadratizable if 
$\alpha f(\y) - f(h(\x))
\geq \beta \left( \bra \F{g}(f, \x), \y - \x \ket + \frac{\mu}{2} \| \y - \x \|^2 \right)$.
Note that this is equivalent to $-\BF{F}$ being upper-quadratizable with the same $\alpha$,$\beta$, and $\mu$, but with $-\F{g}(-f, \x)$.
}

\begin{algorithm2e}[H]
    \SetKwInOut{Input}{Input}\DontPrintSemicolon
    \caption{Online Maximization By Quadratization  - $\mathtt{OMBQ}(\C{A}, \C{G}, h)$}\label{alg:max-by-quad}
    \small
    \Input{ horizon $T$, semi-bandit algorithm $\C{A}$, query algorithm $\C{G}$ for $\F{g}$, the map $h : \C{K} \to \C{K}$ }
    \For{$t = 1, 2, \dots, T$}{
        Play $h(\x_t)$ where $\x_t$ is the action chosen by $\C{A}$ \;
        The adversary selects $f_t$ and a first order query oracle for $f_t$ \;
        Run $\C{G}$ with access to $\x_t$ and the query oracle for $f_t$ to calculate $\oo_t$ \;
        Return $\oo_t$ as the output of the query oracle to $\C{A}$ \;
    }
\end{algorithm2e}
We say an algorithm $\C{G}$ is a first order query algorithm for $\F{g}$ if, given a point $\x \in \C{K}$ and a first order query oracle for $f$, it returns (a possibly unbiased estimate of) $\F{g}(f, \x)$.
We say $\C{G}$ is bounded by $B_1$ if the output of $\C{G}$ is always within the ball $\B{B}^d_{B_1}(\BF{0})$ and we call it trivial if it simply returns the output of the query oracle at $\x$.

Recall that an online agent $\C{A}$ is composed of action function $\C{A}^\t{action}$ and query function $\C{A}^\t{query}$.
Informally, given an online algorithm $\C{A}$ with semi-bandit feedback, we may think of $\C{A}' := \mathtt{OMBQ}(\C{A}, \C{G}, h)$ as the online algorithm with $(\C{A}')^\t{action} \approx h(\C{A}^\t{action})$ and $(\C{A}')^\t{query} \approx \C{G}$.
As a special case, when $h = \op{Id}$ and $\C{G}$ is trivial, we have $\C{A}' = \C{A}$.

\begin{theorem}\label{thm:main}
Let $\C{A}$ be algorithm for online optimization with semi-bandit feedback.
Also let $\BF{F}$ be a function class over $\C{K}$ that is quadratizable with $\mu \geq 0$ and maps $\F{g} : \BF{F} \times \C{K} \to \B{R}^d$ and $h : \C{K} \to \C{K}$, let $\C{G}$ be a query algorithm for $\F{g}$ and let $\C{A}' = \mathtt{OMBQ}(\C{A}, \C{G}, h)$.
Then the following are true.
\begin{enumerate}[left=0pt, labelwidth=0pt, labelsep=0pt, itemindent=2\parindent, itemsep=0pt, parsep=0pt, topsep=0pt, partopsep=0pt]
\item If $\C{G}$ returns the exact value of $\F{g}$, then we have $
\C{R}_{\alpha, \op{Adv}_1^{\t{f}}(\BF{F})}^{\C{A}'}
\leq \beta \C{R}_{1, \op{Adv}_1^{\t{f}}(\BF{F}_{\mu, \F{g}})}^{\C{A}}$. 
\item On the other hand, if $\C{G}$ returns an unbiased estimate of $\F{g}$ and the output of $\C{G}$ is bounded by $B_1$, then we have $\C{R}_{\alpha, \op{Adv}_1^\t{o}(\BF{F}, B_1)}^{\C{A}'}
\leq \beta \C{R}_{1, \op{Adv}_1^{\t{f}}(\BF{Q}_\mu[B_1])}^{\C{A}}$. 
\end{enumerate}
\end{theorem}

As a special case, when $f$ is concave, we may choose $\alpha = \beta = 1$, $h = \op{Id}$, and $\F{g}(f, \x)$ to be a super-gradient of $f$ at $\x$.
In this case, Theorem~\ref{thm:main} reduces to the concave version of Theorems~2 and~5 in~\cite{pedramfar24_unified_framew_analy_meta_onlin_convex_optim}.

\section{Up-concave function optimization is upper-quadratizable function optimization}

In this section, we study three classes of up-concave functions and show that they are upper-quadratizable.
We further use this property to obtain meta-algorithms that convert algorithms for quadratic optimization into algorithms for up-concave maximization. 

\subsection{Monotone up-concave optimization over general convex sets}\label{sec:dr_mono_general}

For differentiable DR-submodular functions, the following lemma is proven for the case $\gamma = 1$ in Lemma~2 in~\cite{fazel22_fast_first_order_method_monot} and for the case $\mu = 0$ in \cite{hassani17_gradien_method_submod_maxim} (See Inequality~7.5 in the arXiv version).
We show the result for functions that are $\mu$-strongly $\gamma$-weakly up-concave with respect to a specified up-super-gradient field and have curvature bounded by $c$; see Appendix~\ref{app:dr_mono_general:quad} for the proof.
The first-order results in this subsection do not require differentiability. Their zeroth-order applications require a differentiable extension that is up-concave with respect to its own gradient, as detailed in Remark~\ref{rem:smoothing-conversion-details}.

\begin{lemma}\label{lem:dr_mono_general:quad}
Let $f : [0, 1]^d \to \B{R}$ be a non-negative continuous monotone function with curvature bounded by $c$ that is $\mu$-strongly $\gamma$-weakly up-concave with respect to a vector-valued function $\tilde{\nabla} f : [0,1]^d \to \B{R}^d$.
Then, for all $\x, \y \in [0, 1]^d$, we have
\begin{align*}
\frac{\gamma^2}{1 + c\gamma^2} f(\y) - f(\x)
\leq \frac{\gamma}{1 + c\gamma^2} \left( \bra \tilde{\nabla} f(\x), \y - \x \ket - \frac{\mu}{2} \| \y - \x \|^2 \right).
\end{align*}
\end{lemma}

Further, we show that any semi-bandit feedback online linear optimization algorithm for fully adaptive adversary is also an online up-concave optimization algorithm.

\begin{theorem}\label{thm:dr_mono_general:main}
Let $\C{K} \subseteq [0, 1]^d$ be a convex set, let $\mu \geq 0$, $\gamma \in (0, 1]$, $c \in [0, 1]$ and let $\C{A}$ be algorithm for online optimization with semi-bandit feedback.
Also let $\BF{F}$ be an $M_1$-Lipschitz function class over $\C{K}$ where every $f \in \BF{F}$ admits a non-negative continuous monotone extension $\bar f : [0,1]^d \to \B{R}$ with curvature bounded by $c$ that is $\mu$-strongly $\gamma$-weakly up-concave with respect to a specified vector-valued function $\tilde{\nabla}\bar f : [0,1]^d \to \B{R}^d$.
For the deterministic first-order bound, the oracle returns $\tilde{\nabla}\bar f(\x)$ at $\x\in\C{K}$, and we assume $\|\tilde{\nabla}\bar f(\x)\|\leq M_1$. For the stochastic bound, the oracle is an unbiased estimator of this same field on $\C{K}$, with outputs bounded by $B_1$.
Then, for any $B_1 \geq M_1$, we have
\begin{align*}
\C{R}_{\frac{\gamma^2}{1 + c\gamma^2}, \op{Adv}_1^{\t{f}}(\BF{F})}^{\C{A}}
\leq \frac{\gamma}{1 + c\gamma^2} \C{R}_{1, \op{Adv}_1^{\t{f}}(\BF{Q}_\mu[M_1])}^{\C{A}}
,\quad
\C{R}_{\frac{\gamma^2}{1 + c\gamma^2}, \op{Adv}_1^\t{o}(\BF{F}, B_1)}^{\C{A}}
\leq \frac{\gamma}{1 + c\gamma^2} \C{R}_{1, \op{Adv}_1^{\t{f}}(\BF{Q}_\mu[B_1])}^{\C{A}}
\end{align*}
\end{theorem}

\begin{wrapfigure}{r}{0.46\textwidth}
\vspace{-.1in}
\begin{algorithm2e}[H]
    \SetKwInOut{Input}{Input}\DontPrintSemicolon
    \caption{Boosted Query oracle for Monotone up-concave functions over convex sets containing the origin -- $\mathtt{BQM0}$}\label{alg:dr_mono_zero:main}
    \small
    \Input{ First order query oracle, point $\x$ }
    Sample $z \in [0, 1]$ according to Equation~\eqref{eq:dr_mono_zero:law_z} \;
    Return the output of the query oracle at $z * \x$ \;
\end{algorithm2e}
\vspace{-.2in}
\end{wrapfigure}

These results follow immediately from Theorem~\ref{thm:main} and Lemma~\ref{lem:dr_mono_general:quad}.
Note that Lemma~\ref{lem:dr_mono_general:quad} requires the non-negative extension to satisfy the up-concavity inequalities with respect to the specified field throughout $[0,1]^d$, not only on $\C{K}$.

\begin{corollary}
The results of~\cite{hassani17_gradien_method_submod_maxim},~\cite{chen18_onlin_contin_submod_maxim} and~\cite{fazel22_fast_first_order_method_monot} on monotone continuous DR-submodular maximization over general convex sets may be thought of as special cases of Theorem~\ref{thm:dr_mono_general:main} when $\C{A}$ is the online gradient ascent algorithm.
\end{corollary}

\subsection{Monotone up-concave optimization over convex sets containing the origin}
\label{sec:dr_mono_zero}

The following lemma is proven for differentiable DR-submodular functions in Theorem~2 and Proposition~1 of~\cite{zhang22_stoch_contin_submod_maxim}.
The proof works for general up-concave functions as well.
We include a proof in Appendix~\ref{app:dr_mono_zero:quad} for completeness.

\begin{lemma}\label{lem:dr_mono_zero:quad}
Let $f : [0, 1]^d \to \B{R}$ be a non-negative monotone $\gamma$-weakly up-concave differentiable function and let $F : [0, 1]^d \to \B{R}$ be the function defined by
\begin{align*}
F(\x) := \int_0^1 \frac{\gamma e^{\gamma (z - 1)}}{(1 - e^{-\gamma})z} (f(z * \x) - f(\BF{0})) dz.
\end{align*}
Then $F$ is differentiable and, if the random variable $\C{Z} \in [0, 1]$ is defined by the law 
\begin{align}\label{eq:dr_mono_zero:law_z}
\forall z \in [0, 1]
,\quad
\B{P}(\C{Z} \leq z) = \int_0^z \frac{\gamma e^{\gamma (u - 1)}}{1 - e^{-\gamma}} du,
\end{align}
then we have $\B{E}\left[ \nabla f(\C{Z} * \x) \right] = \nabla F(\x)$.
Moreover, we have
$(1 - e^{-\gamma}) f(\y) - f(\x)
\leq
\frac{1 - e^{-\gamma}}{\gamma} \bra \nabla F(\x), \y - \x \ket.
$
\end{lemma}

\begin{theorem}\label{thm:dr_mono_zero:main}
Let $\C{K} \subseteq [0, 1]^d$ be a convex set containing the origin, let $\gamma \in (0, 1]$ and let $\C{A}$ be algorithm for online optimization with semi-bandit feedback.
Also let $\BF{F}$ be a function class over $\C{K}$ where every $f \in \BF{F}$ is the restriction of a non-negative monotone $\gamma$-weakly up-concave function defined over $[0, 1]^d$ to the set $\C{K}$.
Assume $\BF{F}$ is differentiable and $M_1$-Lipschitz for some $M_1 > 0$.
Then, for any $B_1 \geq M_1$, we have
\begin{align*}
\C{R}_{1 - e^{-\gamma}, \op{Adv}_1^\t{o}(\BF{F}, B_1)}^{\C{A}'}
\leq \frac{1 - e^{-\gamma}}{\gamma} \C{R}_{1, \op{Adv}_1^{\t{f}}(\BF{Q}_0[B_1])}^{\C{A}}
\vspace{-.1in}
\end{align*}
where $\C{A}' = \mathtt{OMBQ}(\C{A}, \mathtt{BQM0}, \op{Id})$.
\end{theorem}

This result now follows immediately from Theorem~\ref{thm:main} and Lemma~\ref{lem:dr_mono_zero:quad}.

\begin{corollary}\label{cor:dr_mono_zero:boosting}
The result of~\cite{zhang22_stoch_contin_submod_maxim} in the online setting (when there is no delay) may be seen as an application of Theorem~\ref{thm:dr_mono_zero:main} when $\C{A}$ is chosen to be online gradient ascent.
\end{corollary}

\subsection{Non-monotone up-concave optimization over general convex sets}

The following lemma is proven for differentiable DR-submodular functions in Corollary~2, Theorem~4 and Proposition~2 of~\cite{zhang24_boost_gradien_ascen_contin_dr_maxim}.
The arguments works for general up-concave functions as well.
We include a proof in Appendix~\ref{app:dr_nonmono:quad} for completeness.

\begin{lemma}\label{lem:dr_nonmono:quad}
Let $f : [0, 1]^d \to \B{R}$ be a non-negative continuous up-concave differentiable function and let $\underline{\x} \in \C{K}$.
Define $F : [0, 1]^d \to \B{R}$ as the function
$F(\x) := \int_0^1 \frac{2}{3 z (1 - \frac{z}{2})^3} \left( f\left(\frac{z}{2} * (\x - \underline{\x}) + \underline{\x} \right) - f(\underline{\x}) \right) dz$. 
Then $F$ is differentiable and, if the random variable $\C{Z} \in [0, 1]$ is defined by the law 
\begin{align}\label{eq:dr_nonmono:law_z}
\forall z \in [0, 1]
,\quad
\B{P}(\C{Z} \leq z) = \int_0^z \frac{1}{3 (1 - \frac{u}{2})^3} du,
\end{align}
then we have 
$\B{E}\left[ \nabla f\left(\frac{\C{Z}}{2} * (\x - \underline{\x}) + \underline{\x} \right) \right] = \nabla F(\x)$.
Moreover, we have
\begin{align*}
\frac{1 - \|\underline{\x}\|_\infty}{4} f(\y) - f\left(\frac{\x + \underline{\x}}{2} \right)
&\leq
\frac{3}{8} \bra \nabla F(\x), \y - \x \ket.
\end{align*}
\end{lemma}

\begin{theorem}\label{thm:dr_nonmono:main}
Let $\C{K} \subseteq [0, 1]^d$ be a convex set, $\underline{\x} \in \C{K}$, $h := \|\underline{\x}\|_\infty$ and $\C{A}$ be an algorithm for online optimization with semi-bandit feedback.
Also let $\BF{F}$ be a function class over $\C{K}$ where every $f \in \BF{F}$ is the restriction of a non-negative up-concave function defined over $[0, 1]^d$ to the set $\C{K}$.
Assume $\BF{F}$ is differentiable and $M_1$-Lipschitz for some $M_1 > 0$.
Then, for any $B_1 \geq M_1$ and $\C{A}' = \mathtt{OMBQ}(\C{A}, \mathtt{BQN}, \x \mapsto \frac{\x + \underline{\x}}{2})$, we have $\C{R}_{\frac{1 - h}{4}, \op{Adv}_1^\t{o}(\BF{F}, B_1)}^{\C{A}'}
\leq \frac{3}{8} \C{R}_{1, \op{Adv}_1^{\t{f}}(\BF{Q}_0[B_1])}^{\C{A}}$. 
\end{theorem}

\begin{algorithm2e}[H]
    \SetKwInOut{Input}{Input}\DontPrintSemicolon
    \caption{Boosted Query oracle for Non-monotone up-concave functions over general convex sets -- $\mathtt{BQN}$}\label{alg:dr_nonmono:main}
    \small
    \Input{ First order query oracle, point $\x$, fixed reference point $\underline{\x} \in \C{K}$ }
    Sample $z \in [0, 1]$ according to Equation~\ref{eq:dr_nonmono:law_z} \;
    Return the output of the query oracle at $\frac{z}{2} * (\x - \underline{\x}) + \underline{\x}$ \;
\end{algorithm2e}

These results now follow immediately from Theorem~\ref{thm:main} and Lemma~\ref{lem:dr_nonmono:quad}.

\begin{corollary}\label{cor:dr_nonmono:boosting}
The result of~\cite{zhang24_boost_gradien_ascen_contin_dr_maxim} in the online setting without delay may be seen as an application of Theorem~\ref{thm:dr_nonmono:main} when $\C{A}$ is chosen to be online gradient ascent.
\end{corollary}

\section{Meta algorithms for other feedback cases}\label{sec:meta}

In this section, we study several meta-algorithms that allow us to convert between different feedback types and also convert results from the online setting to the offline setting.

{\bf First order/semi-bandit to zeroth order/bandit feedback: }
In this section we discuss meta-algorithms that convert algorithms designed for first order feedback into algorithms that can handle zeroth order feedback.
These algorithms and results are generalization of similar results in~\cite{pedramfar24_unified_framew_analy_meta_onlin_convex_optim} to the case where $\alpha < 1$.

We choose a point $\BF{c} \in \op{relint}(\C{K})$ and a real number $r > 0$ such that $\op{aff}(\C{K}) \cap \B{B}_r(\BF{c}) \subseteq \C{K}$.
Then, for any shrinking parameter $0 \leq \delta < r$, we define
\[
T_\delta : \B{R}^d \to \B{R}^d := \x \mapsto \left(1 - \frac{\delta}{r}\right) \x + \frac{\delta}{r} \BF{c}.
\]
For a function $f : \C{K} \to \B{R}$ defined on a convex set $\C{K} \subseteq \B{R}^d$, its $\delta$-smoothed version $\hat{f}_\delta$ is a function from $\hat{\C{K}}_\delta := T_\delta(\C{K}) = (1 - \frac{\delta}{r}) \C{K} + \frac{\delta}{r} \BF{c}$ to $\B{R}$ given by
\begin{align*}
\hat{f}_\delta(\x) 
&:= \B{E}_{\z \sim \op{aff}(\C{K}) \cap \B{B}_\delta(\x)}[f(\z)] 
= \B{E}_{\vv \sim \C{L}_0 \cap \B{B}_1(\BF{0})}[f(\x + \delta \vv)],
\end{align*}
where $\C{L}_0 = \op{aff}(\C{K}) - \x$, for any $\x \in \C{K}$, is the linear space that is a translation of the affine hull of $\C{K}$ and $\vv$ is sampled uniformly at random from the $k = \op{dim}(\C{L}_0)$-dimensional ball $\C{L}_0 \cap \B{B}_1(\BF{0})$.
Thus, the function value $\hat{f}_\delta(\x)$ is obtained by ``averaging'' $f$ over a sliced ball of radius $\delta$ around $\x$.
Note that the domain of the function $\hat{f}_\delta$ is not the same as $f$.
To make the domains equal, we define the function $f'_\delta : \C{K} \to \B{R}$ with
\[
f'_\delta(\x) := \hat{f}_\delta( T_\delta( \x )).
\]
For a function class $\BF{F}$ over $\C{K}$, we use $\BF{F}'_\delta$ to denote $\{ f'_\delta = \hat{f}_\delta \circ T_\delta \mid f \in \BF{F} \}$.
We will drop the subscript $\delta$ when there is no ambiguity (See Appendix~\ref{app:first-order-to-zero-order} for the description of the algorithms and the proof.).

\begin{theorem}\label{thm:first-order-to-zero-order}
Let $\BF{F}$ be an $M_1$-Lipschitz function class over a convex set $\C{K}$, let $D = \op{diam}(\C{K})$, choose $\BF{c}$ and $r$ as described above, and let $0 < \delta < r$.
Let $\C{U} \subseteq \C{K}^T$ be a compact set.
Assume $\C{A}$ is an algorithm for online optimization with first order feedback.
Then, if $\C{A}' = \mathtt{FOTZO}(\C{A})$ where $\mathtt{FOTZO}$ is described by Algorithm~\ref{alg:first-order-to-zeroth-order} and $0 < \alpha \leq 1$, we have
\begin{align*}
\C{R}_{\alpha, \op{Adv}^\t{o}_0(\BF{F}, B_0)}^{\C{A}'}(\C{U})
\leq \C{R}_{\alpha, \op{Adv}^\t{o}_1(\BF{F}', \frac{k}{\delta}B_0)}^{\C{A}}(\C{U})
+ \left( 3 + \frac{2 D}{r} \right) \delta M_1 T.
\end{align*}
On the other hand, if we assume that $\C{A}$ is semi-bandit, then the same regret bounds hold with $\C{A}' = \mathtt{STB}(\C{A})$, where $\mathtt{STB}$ is described by Algorithm~\ref{alg:semi-bandit-to-bandit}. 
\end{theorem}

\begin{theorem}\label{thm:first-order-to-det-zero-order}
Under the assumptions of Theorem~\ref{thm:first-order-to-zero-order}, if $\C{A}' = \mathtt{FOTZO\textrm{-}2P}(\C{A})$ where $\mathtt{FOTZO\textrm{-}2P}$ is described by Algorithm~\ref{alg:first-order-to-det-zeroth-order} and $0 < \alpha \leq 1$, we have
\begin{align*}
\C{R}_{\alpha, \op{Adv}^\t{o}_0(\BF{F})}^{\C{A}'}(\C{U})
\leq \C{R}_{\alpha, \op{Adv}^\t{o}_1(\BF{F}', k M_1)}^{\C{A}}(\C{U})
+ \left( 3 + \frac{2 D}{r} \right) \delta M_1 T.
\end{align*}
\end{theorem}

\begin{remark}\label{rem:smoothing-conversion-details}
Theorems~\ref{thm:first-order-to-zero-order} and~\ref{thm:first-order-to-det-zero-order} require a regret bound for the base algorithm over $\BF{F}'_\delta$, not merely over $\BF{F}$.
For an arbitrary function class, closure under smoothing and rescaling must be assumed or proved; it does not follow solely from the properties of individual members.

For the applications to up-concave functions, suppose $f$ has a differentiable Lipschitz extension $\bar f$ to $[0,1]^d$ satisfying the corresponding nonnegativity, monotonicity, and curvature assumptions, and that $\bar f$ is $\mu$-strongly $\gamma$-weakly up-concave with respect to its own gradient $\nabla\bar f$.
Up-concavity with respect to an unrelated vector field alone is not sufficient for these applications of $\mathtt{FOTZO}$, $\mathtt{STB}$, or $\mathtt{FOTZO\textrm{-}2P}$: their simulated first-order oracles estimate the gradient of the smoothed function (or its projection onto the feasible affine hull), not an arbitrary prescribed up-super-gradient field.
Set $s=1-\delta/r$ and define
\[
\bar f'_\delta(\x)
:=\B{E}_{\vv\sim\C{L}_0\cap\B{B}_1(\BF{0})}
[\bar f(s\x+(1-s)\BF{c}+\delta\vv)],
\qquad \x\in[0,1]^d.
\]
This is well-defined: $\BF{c}+r\vv\in\C{K}\subseteq[0,1]^d$, so the argument is the convex combination $s\x+(1-s)(\BF{c}+r\vv)$ of two points in the cube.
Its restriction to $\C{K}$ is exactly $f'_\delta$.
Differentiating under the expectation gives
\[
\nabla\bar f'_\delta(\x)
=s\B{E}[\nabla\bar f(s\x+(1-s)\BF{c}+\delta\vv)].
\]
Applying the up-concavity inequalities inside this expectation shows that $\bar f'_\delta$ is $s^2\mu$-strongly $\gamma$-weakly up-concave if $\bar f$ is $\mu$-strongly $\gamma$-weakly up-concave.
Nonnegativity and, when applicable, monotonicity are preserved.
The curvature bound $c$ is also preserved: apply its increment inequality to each translated and scaled pair of points, with increment $s\z$, and then take expectations.
On $\C{K}$, $f'_\delta$ is $sM_1$-Lipschitz whenever $f$ is $M_1$-Lipschitz.
For a lower-dimensional feasible set, the estimators return gradients along its affine hull; orthogonal projection onto $\C{L}_0$ does not change their inner products with feasible differences $\y-\x$.

Consequently, the applications below apply the first-order guarantees to the class of all functions with the corresponding properties, with strong up-concavity parameter $s^2\mu$ (or simply $0$ when using a linear base algorithm).
This justifies the zeroth-order conversions without asserting $\BF{F}'_\delta\subseteq\BF{F}$ for an arbitrary subclass.
The feasible set remains $\C{K}$, so the assumption $0\in\C{K}$ for $\mathtt{BQM0}$ is preserved.
For the $\mathtt{SFTT}$ conversions, the class of all functions satisfying the corresponding assumptions is closed under convex combinations, by averaging their extensions and defining inequalities. In the generalized first-order setting, the specified up-super-gradient fields are averaged together with the function extensions.
\end{remark}

{\bf Full information to trivial query: }
In this section, we discuss a meta-algorithm that converts algorithms that require full-information feedback into algorithms that have a trivial query oracle.
In particular, it converts algorithms that require first-order full-information feedback into semi-bandit algorithms and algorithms that require zeroth-order full-information feedback into bandit algorithms.

Here we assume that $\C{A}^\t{query}$ does not depend on the observations in the current round.
If the number of queries $k_t$ is not constant for each time-step, we simply assume that $\C{A}$ queries extra points and then discards them, so that we obtain an algorithm that queries exactly $K$ points at each time-step, where $K$ does not depend on $t$.  We say a function class $\BF{F}$ is closed under convex combination if for any $f_1, \cdots, f_k \in \BF{F}$ and any $\delta_1, \cdots, \delta_k \geq 0$ with $\sum_{i} \delta_i = 1$, we have $\sum_i \delta_i f_i \in \BF{F}$.

\begin{theorem}\label{thm:stoch-fi-to-sb}
Let $\C{A}$ be an online optimization algorithm with full-information feedback and $K$ queries per round, chosen independently of the observations in that round, and let $\C{A}'=\mathtt{SFTT}(\C{A})$.
Let $\BF{F}$ be a class of $M_1$-Lipschitz functions on $\C{K}$ that is closed under convex combinations, let $i\in\{0,1\}$, $B>0$, $0<\alpha\leq1$, and $D=\op{diam}(\C{K})$.
When $\alpha<1$, assume additionally that the functions are non-negative.
Choose an integer block length $L>K$ dividing $T$ and let $1\leq a\leq b\leq T$.
Set $a'=\lceil(a-1)/L\rceil+1$, $b'=\lfloor b/L\rfloor$, and $N_{a,b}=\max\{0,b'-a'+1\}$, so that $[a',b']$ indexes the complete blocks contained in $[a,b]$.
Writing $\{T\}$ and $\{T/L\}$ for the adversary horizons, we have
\begin{align*}
\C{R}_{\alpha,\op{Adv}_i^\t{o}(\BF{F},B)\{T\}}^{\C{A}'}(\C{K}_\star^T)[a,b]
&\leq M_1DKN_{a,b}+2M_1D(L-1)\\
&\quad+L\C{R}_{\alpha,\op{Adv}_i^\t{o}(\BF{F},B)\{T/L\}}^{\C{A}}(\C{K}_\star^{T/L})[a',b'].
\end{align*}
The final term is defined to be zero when $a'>b'$.
The conversion yields semi-bandit feedback when $i=1$ and bandit feedback when $i=0$.
\end{theorem}

This result (proof in Appendix \ref{app:stoch-fi-to-sb}) is based on the idea of random permutations used in~\cite{zhang19_onlin_contin_submod_maxim,zhang23_onlin_learn_non_submod_maxim,pedramfar23_unified_projec_free_algor_adver}.

\textbf{Online to Offline:}  An offline optimization problem can be though of as an instance of online optimization where the adversary picks the same function and query oracle at each round.
Moreover, instead of regret, the performance of the algorithm is measured by sample complexity, i.e., the minimum number of queries required so that the expected error from the $\alpha$-approximation of the optimal value is less than $\epsilon$.

\begin{algorithm2e}[H]
\SetKwInOut{Input}{Input}\DontPrintSemicolon
\caption{Stochastic Full-information To Trivial query - $\mathtt{SFTT}(\C{A})$}
\label{alg:stoch-fi-to-sb}
\small
\Input{base algorithm $\C{A}$ run for $T/L$ rounds, horizon $T$, integer block size $L>K$ dividing $T$.}
\For{$q = 1, 2, \dots, T/L$}{
    Let $\hat{\x}_q$ be the action chosen by $\C{A}^\t{action}$ \;
    Let $(\hat{\y}_q^i)_{i = 1}^{K}$ be the queries selected by $\C{A}^\t{query}$ \;
    Let $(t_{q,1}, \dots, t_{q,L})$ be a random permutation of $\{(q-1)L+1, \dots, qL\}$\;
    \For{$t = (q-1)L+1, \dots, qL$}{
        \eIf{$t = t_{q, i}$ for some $1 \leq i \leq K$}{
            Play the action $\x_t = \hat{\y}_q^i$ \;
            Store the observation as the response to the $i$-th query \;
        }{
            Play the action $\x_t = \hat{\x}_q$ \;
        }
    }
    Pass the stored responses to $\C{A}$ in query order and complete its update for round $q$ \;
}
\end{algorithm2e}

Conversions of online algorithms to offline are referred to online-to-batch techniques and are well-known in the literature (See~\cite{shalev-shwartz12_onlin_learn_onlin_convex_optim}). 
A simple approach is to simply run the online algorithm and if the actions chosen by the algorithm are $\x_1, \cdots, \x_T$, return $\x_t$ for $1 \leq t \leq T$ with probability $1/T$.
We use $\mathtt{OTB}$ to denote the meta-algorithm that uses this approach to convert online algorithms to offline algorithms.
The following theorem is a corollary which we include for completion. (See Appendix~\ref{app:online-to-offline} for the proof.)

\begin{theorem}\label{thm:online-to-offline}
Let $\C{A}$ be an online algorithm that queries no more than $K = T^\theta$ times per time-step that obtains an $\alpha$-regret bound of $O(T^{\eta})$ over an oblivious adversary $\op{Adv}$.
Then the sample complexity of $\mathtt{OTB}(\C{A})$ over $\{ (f, \C{Q}_f) \mid ((f, \C{Q}_f), \cdots, (f, \C{Q}_f)) \in \op{Adv}\}$ is $O(\epsilon^{-\frac{1 + \theta}{1 - \eta}})$.
\end{theorem}

\section{Applications}\label{sec:applications}

Figure~\ref{fig:main} captures the applications that are mentioned in Tables~\ref{tbl:adv},~\ref{tbl:offline} and~\ref{tbl:non-stationary}.
The exact statements are stated in Corollaries~\ref{cor:w-ada-reg} and~\ref{cor:d-reg} in the Appendix.
To obtain a result from the graph, let $\C{A}$ be one of $\mathtt{SO\textrm{-}OGA}$ or $\mathtt{IA}$ and select a directed path that has the following properties: 
\begin{enumerate}
\item[(i)] The path starts at one of the three nodes on the left. 
\item[(ii)]  The path must be at least of length 1 and the edges must be the same color.
\item[(iii)] If $\C{A}$ is $\mathtt{IA}$, the path should not contain $\mathtt{SFTT}$ or $\mathtt{OTB}$.
\end{enumerate}

\begin{figure}
\centering
\caption{\small Summary of applications (See Section~\ref{sec:applications}) }
\resizebox{.8\textwidth}{!}{\includegraphics[trim= 1.9in 0in 0in 0in,clip]{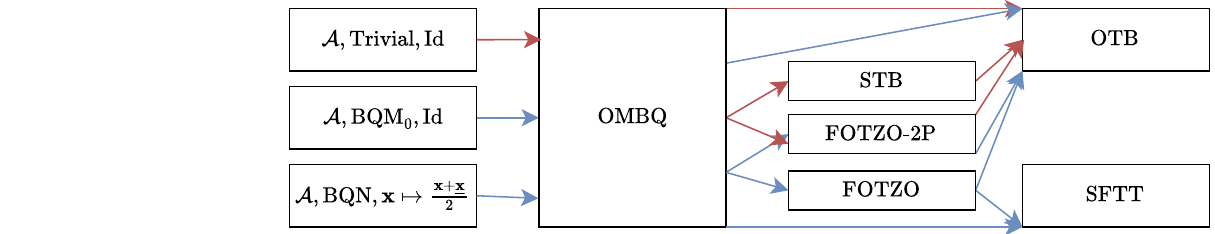}}
\label{fig:main}
\end{figure}

For example, if $\C{A} = \mathtt{SO\textrm{-}OGA}$ and the path starts at the middle node on the left, then passes through $\mathtt{OMBQ}$, $\mathtt{FOTZO}$, $\mathtt{SFTT}$, we get $\mathtt{SFTT}(
    \mathtt{FOTZO}(
        \mathtt{OMBQ}(
            \mathtt{SO\textrm{-}OGA}, \mathtt{BQM0}, \op{Id}
        )
    )
)$, 
which is a projection-free algorithm (using separation oracles) with bandit feedback for monotone up-concave functions over convex sets containing the origin.
As mentioned in Table~\ref{tbl:non-stationary} and Corollary~\ref{cor:w-ada-reg}-(b), the adaptive regret of this algorithm is of order $O(T^{4/5})$.
Note that the text written in the three nodes on the left correspond to the inputs of the meta-algorithm $\mathtt{OMBQ}$.
Also note that the color red corresponds to the setting where $\C{G}$ is a trivial query algorithm which means that the output of $\mathtt{OMBQ}$ is semi-bandit.

\section{Conclusions}

In this work, we have presented a comprehensive framework for addressing optimization problems involving upper-quadratizable functions, encompassing both concave and DR-submodular functions across various settings and feedback types. Our contributions include the formulation of upper-quadratizable functions as a generalized class, the development of meta-algorithms for algorithmic conversions, and the derivation of new  algorithms with improved static/ dynamic/ adaptive regret guarantees. Exploring more subset of classes of upper-quadratizable functions where such a framework could be applied is an important future direction. 

\paragraph{Subsequent developments.}
Following the original NeurIPS 2024 publication of this work, the framework has been extended in several directions. \citet{pedramfar25_uniform_wrappers} developed uniform wrappers for transferring algorithms and regret analyses from concave to quadratizable online optimization. \citet{lu26_upper_linearizability_down_closed} studied online non-monotone DR-submodular maximization over down-closed convex sets through upper-linearizability. Beyond up-concave functions, \citet{pedramfar26_weakly_up_concavity} introduced $\gamma$-weakly $\theta$-up-concavity, extending the applicability of linearizable optimization to a broader class encompassing DR-submodular and one-sided smooth (OSS) functions. A decentralized projection-free extension for online upper-linearizable optimization, with applications to DR-submodular optimization, was developed by \citet{lu25_decentralized_upper_linearizable}.

\section{Acknowledgement}

This research was supported in part by the National Science Foundation under grant CCF-2149588.

\newpage

\bibliography{references}

\newpage

\appendix

\section{Related works}\label{sec:related}

\paragraph{DR-submodular maximization}

Two of the main methods for continuous DR-submodular maximization are
\emph{Frank-Wolfe type methods} and \emph{Boosting based methods}.
This division is based on how the approximation coefficient appears in the proof.

In Frank-Wolfe type algorithms, the approximation coefficient appears by specific choices of the Frank-Wolfe update rules. (See Lemma~8 in~\cite{pedramfar23_unified_projec_free_algor_adver})
The specific choices of the update rules for different settings have been proposed in~\cite{bian17_guaran_non_optim,bian17_contin_dr_maxim,mualem22_resol_approx_offlin_onlin_non,pedramfar23_unified_approac_maxim_contin_dr_funct,chen23_contin_non_dr_maxim_down_convex_const}.
The momentum technique of~\cite{mokhtari20_stoch_condit_gradien_method} has been used to convert algorithms designed for deterministic feedback to stochastic feedback setting.
\cite{hassani20_stoch_condit_gradien} proposed a Frank-Wolfe variant with access to a stochastic gradient oracle with \emph{known distribution}.
Frank-Wolfe type algorithms been adapted to the online setting using Meta-Frank-Wolfe~\cite{chen18_onlin_contin_submod_maxim,chen19_projec_free_bandit_convex_optim} or using Blackwell approachablity~\cite{niazadeh21_onlin_learn_offlin_greed_algor}.
Later \cite{zhang19_onlin_contin_submod_maxim} used a Meta-Frank-Wolfe with random permutation technique to obtain full-information results that only require a single query per function and also bandit results.
This was extended to another settings by~\cite{zhang23_onlin_learn_non_submod_maxim} and generalized to many different settings with improved regret bounds by~\cite{pedramfar23_unified_projec_free_algor_adver}.

Another approach, referred to as boosting, is to construct an alternative function such that maximization of this function results in approximate maximization of the original function.
Given this definition, we may consider the result of~\cite{hassani17_gradien_method_submod_maxim,chen18_onlin_contin_submod_maxim,fazel22_fast_first_order_method_monot} as the first boosting based results.
However, in these cases (i.e., the case of monotone DR-submodular functions over general convex sets), the alternative function is identical to the original function.
The term boosting in this context was first used in~\cite{zhang22_stoch_contin_submod_maxim} for monotone functions over convex sets containing the origin, based on ideas presented in~\cite{filmus12_tight_combin_algor_submod_maxim,mitra2021submodular}.
This idea was used later in~\cite{wan23_bandit_multi_dr_submod_maxim,liao23_improv_projec_onlin_contin_submod_maxim} in bandit and projection-free full-information settings.
Finally, in~\cite{zhang24_boost_gradien_ascen_contin_dr_maxim} a boosting based method was introduced for non-monotone functions over general convex sets.

\paragraph{Up-concave maximization}

Not all continuous DR-submodular functions are concave and not all concave functions are continuous DR-submodular.
\cite{mitra2021submodular} considers functions that are the sum of a concave and a continuous DR-submodular function.
It is well-known that continuous DR-submodular functions are concave along positive directions~\cite{calinescu11_maxim_monot_submod_funct_subjec_matroid_const,bian17_guaran_non_optim}.
Based on this idea, \cite{wilder18_equil} defined an up-concave function as a function that is concave along positive directions.
Up-concave maximization has been considered in the offline setting before, e.g.~\cite{lee23_non_smoot_hextb_smoot_robus_submod_maxim}, but not in online setting.
In this work, we focus on up-concave maximization which is a generalization of DR-submodular maximization.

\paragraph{Projection-free optimization}

In the past decade, numerous projection-free online convex optimization algorithms have emerged to tackle the computational limitations of their projection-based counterparts \cite{hazan12_projec,chen18_projec_free_onlin_optim_stoch_gradien,xie20_effic_projec_free_onlin_method,chen19_projec_free_bandit_convex_optim,hazan20_faster_projec_onlin_learn,garber20_improv_regret_bound_projec_bandit_convex_optim,mhammedi22_effic_projec_free_onlin_convex,garber22_new_projec_algor_onlin_convex}. 
In the context of DR-submodular maximization, the Frank-Wolfe type methods discussed above are projection-free.

\paragraph{Non-stationary regret}

Dynamic regret was first analyzed in~\cite{zinkevich03_onlin} for first order deterministic feedback.
Later~\cite{zhang18_adapt_onlin_learn_dynam_envir} obtained the lower bound and optimal algorithm in this setting.
This was later expanded to bandit setting in~\cite{zhao21_bandit_convex_optim_non_envir}.
Adaptive regret was first analyzed in~\cite{hazan09_effic} and the first optimal algorithm for projection-free adaptive regret was proposed in~\cite{garber22_new_projec_algor_onlin_convex}.
We refer to ~\cite{hazan09_effic,besbes15_non_station_stoch_optim,daniely15_stron_adapt_onlin_learn,zhang18_dynam_regret_stron_adapt_method,zhang18_adapt_onlin_learn_dynam_envir,zhao20_dynam_regret_convex_smoot_funct,zhao21_bandit_convex_optim_non_envir,lu23_projec_adapt_regret_member_oracl,wang24_non_projec_free_onlin_learn,garber22_new_projec_algor_onlin_convex} and references therein for more details.

\paragraph{Optimization by quadratization}

The framework discussed here for analyzing online algorithms is based on the convex optimization framework introduced in~\cite{pedramfar24_unified_framew_analy_meta_onlin_convex_optim}.
We extend the framework to allows us to work with $\alpha$-regret.Moreover, \cite{pedramfar24_unified_framew_analy_meta_onlin_convex_optim} also demonstrates that algorithms that are designed for quadratic/linear optimization with fully adaptive adversary obtain a similar regret in the convex setting.
In this paper we introduce the notion of quadratizable functions generalizes this idea beyond convex functions to all quadratizable functions. (see Theorem~\ref{thm:main})
This allows us to integrate the boosting method with our framework to obtain various meta-algorithms for continuous DR-submodular maximization.

\section{Problem Setup in Detail}\label{apdx:setup}

In this section, we further expand on the description in Section \ref{main_def}. 

A \textit{function class} is  a set of real-valued functions.
Given a set $\C{D}$, a \textit{function class} over $\C{D}$ is  a subset of all real-valued functions on $\C{D}$.
A set $\C{K} \subseteq \B{R}^d$ is called a \textit{convex set} if for all $\x,\y \in \C{K}$ and $\alpha \in [0, 1]$, we have $\alpha \x + (1 - \alpha) \y \in \C{K}$.
For any $\uu \in \C{K}^T$, we define the path length 
$P_T(\uu) := \sum_{i = 1}^{T-1} \| \uu_i - \uu_{i+1} \|.$

A real-valued differentiable function $f$ is called concave if 
$f(y) - f(x) \leq f'(x) (y - x)$,
for all $x, y \in \op{Dom}(f)$.
More generally, given $\mu \geq 0$ and $0 < \gamma \leq 1$, we say a real-valued differentiable function is \textit{$\mu$-strongly $\gamma$-weakly concave} if
\[
\gamma \left( f'(y) (y - x) + \frac{\mu}{2} | y - x |^2 \right)
\leq
f(y) - f(x)
\leq \frac{1}{\gamma} \left( f'(x) (y - x) - \frac{\mu}{2} | y - x |^2 \right)
\]
for all $x \le y$ in $\op{Dom}(f)$.

We say a differentiable function $f : \C{K} \to \B{R}$ is \textit{$\mu$-strongly $\gamma$-weakly up-concave} if it is $\mu$-strongly $\gamma$-weakly concave along positive directions.
Specifically if, for all $\x \leq \y$ in $\C{K}$, we have
\[
\gamma \left( \bra \nabla f(\y), \y - \x \ket + \frac{\mu}{2} \| \y - \x \|^2 \right)
\leq f(\y) - f(\x)
\leq \frac{1}{\gamma} \left( \bra \nabla f(\x), \y - \x \ket - \frac{\mu}{2} \| \y - \x \|^2 \right).
\]
This notion could be generalized in the following manner.
We say $\tilde{\nabla} f : \C{K} \to \B{R}^d$ is a \textit{$\mu$-strongly $\gamma$-weakly up-super-gradient} of $f$ if for all $\x \leq \y$ in $\C{K}$, we have
\[
\gamma \left( \bra \tilde{\nabla} f(\y), \y - \x \ket + \frac{\mu}{2} \| \y - \x \|^2 \right)
\leq f(\y) - f(\x)
\leq \frac{1}{\gamma} \left( \bra \tilde{\nabla} f(\x), \y - \x \ket - \frac{\mu}{2} \| \y - \x \|^2 \right).
\]
We say a differentiable function $f$ is $\mu$-strongly $\gamma$-weakly up-concave if $\nabla f$ is a $\mu$-strongly $\gamma$-weakly up-super-gradient for $f$.
More generally, we say $f$ is $\mu$-strongly $\gamma$-weakly up-concave with respect to the vector-valued function $\tilde{\nabla} f : \C{K} \to \B{R}^d$ if $\tilde{\nabla} f$ is a $\mu$-strongly $\gamma$-weakly up-super-gradient for $f$.
Unless a vector field is explicitly specified, up-concavity refers to the differentiable case with respect to the function's own gradient.
When $\gamma = 1$ and the above inequality holds for all $\x, \y \in \C{K}$, we say $f$ is $\mu$-strongly concave.

A differentiable function $f : \C{K} \to \B{R}$ is called \textit{continuous DR-submodular} if for all $\x \leq \y$, we have $\nabla f(\x) \geq \nabla f(\y)$.
More generally, we say $f$ is \textit{$\gamma$-weakly continuous DR-submodular} if for all $\x \leq \y$, we have $\nabla f(\x) \geq \gamma \nabla f(\y)$.
It follows that any $\gamma$-weakly continuous DR-submodular functions is $\gamma$-weakly up-concave.

Given a continuous monotone function $f : \C{K} \to \B{R}$, its curvature is defined as the smallest number $c \in [0, 1]$ such that
\[
f(\y + \z) - f(\y) \geq (1 - c) (f(\x + \z) - f(\x)),
\]
for all $\x, \y \in \C{K}$ and $\z \geq 0$ such that $\x + \z, \y + \z \in \C{K}$.
\footnote{
In the literature, the curvature is often defined for differentiable functions.
When $f$ is differentiable, we have
\[
c = 1 - \inf_{\x, \y \in \C{K}, 1 \leq i \leq d} \frac{[\nabla f(\y)]_i}{[\nabla f(\x)]_i}.
\]
}
We define the curvature of a function class $\BF{F}$ as the supremum of the curvature of functions in $\BF{F}$.

Online optimization problems can be formalized as a repeated game between an agent and an adversary.
The game lasts for $T$ rounds on a convex domain $\C{K}$ where $T$ and $\C{K}$ are known to both players.
In $t$-th round, the agent chooses an action $\x_t$ from an action set $\C{K} \subseteq \B{R}^d$, then the adversary chooses a loss function $f_t \in \BF{F}$ and a query oracle for the function $f_t$.
Then, for $1 \leq i \leq k_t$, the agent chooses a points $\y_{t, i}$ and receives the output of the query oracle.
Here $k_t$ denotes the total number of queries made by the agent at time-step $t$, which may or may not be known in advance.

To be more precise, an agent consists of a tuple $(\Omega^\C{A}, \C{A}^{\t{action}}, \C{A}^{\t{query}})$, where $\Omega^\C{A}$ is a probability space that captures all the randomness of $\C{A}$.
We assume that, before the first action, the agent samples $\omega \in \Omega$.
The next element in the tuple, $\C{A}^{\t{action}} = (\C{A}^{\t{action}}_1, \cdots, \C{A}^{\t{action}}_T)$ is a sequence of functions such that $\C{A}_t$ that maps the history $\Omega^\C{A} \times \C{K}^{t-1} \times \prod_{s = 1}^{t-1} (\C{K} \times \C{O})^{k_s}$ to $\x_t \in \C{K}$ where we use $\C{O}$ to denote range of the query oracle.
The last element in the tuple, $\C{A}^{\t{query}}$, is the query policy.
For each $1 \leq t \leq T$ and $1 \leq i \leq k_t$, $\C{A}^{\t{query}}_{t, i} : \Omega^\C{A} \times \C{K}^t \times \prod_{s = 1}^{t-1} (\C{K} \times \C{O})^{k_s} \times (\C{K} \times \C{O})^{i-1}$ is a function that, given previous actions and observations, either selects a point $\y_t^i \in \C{K}$, i.e., query, or signals that the query policy at this time-step is terminated.
We may drop $\omega$ as one of the inputs of the above functions when there is no ambiguity.
We say the agent query function is \textit{trivial} if $k_t = 1$ and $\y_{t, 1} = \x_t$ for all $1 \leq t \leq T$.
In this case, we simplify the notation and use the notation $\C{A} = \C{A}^{\t{action}} = (\C{A}_1, \cdots, \C{A}_T)$ to denote the agent action functions and assume that the domain of $\C{A}_t$ is $\Omega^\C{A} \times (\C{K} \times \C{O})^{t-1}$.

A query oracle is a function that provides the observation to the agent.
Formally, a query oracle for a function $f$ is a map $\C{Q}$ defined on $\C{K}$ such that for each $\x \in \C{K}$, the $\C{Q}(\x)$ is a random variable taking value in the observation space $\C{O}$.
The query oracle is called a \textit{stochastic value oracle} or \textit{stochastic zeroth order oracle} if $\C{O} = \B{R}$ and $f(\x) = \B{E}[\C{Q}(\x)]$.
Similarly, it is called a \textit{stochastic up-super-gradient oracle} or \textit{stochastic first order oracle} if $\C{O} = \B{R}^d$ and $\B{E}[\C{Q}(\x)]$ is a up-super-gradient of $f$ at $\x$.
When up-concavity is specified with respect to a field $\tilde{\nabla} f$, the corresponding oracle satisfies $\B{E}[\C{Q}(\x)]=\tilde{\nabla} f(\x)$. For differentiable up-concave functions with no other field specified, we use gradient oracles, for which $\B{E}[\C{Q}(\x)]=\nabla f(\x)$.
In all cases, if the random variable takes a single value with probability one, we refer to it as a \textit{deterministic} oracle.
Note that, given a function, there is at most a single deterministic gradient oracle, but there may be many deterministic up-super-gradient oracles.
We will use $\nabla$ to denote the deterministic gradient oracle.
We say an oracle is bounded by $B$ if its output is always within the Euclidean ball of radius $B$ centered at the origin.
We say the agent takes \textit{semi-bandit feedback} if the oracle is first-order and the agent query function is trivial.
Similarly, it takes \textit{bandit feedback} if the oracle is zeroth-order and the agent query function is trivial.
\footnote{This is a slight generalization of the common use of the term bandit feedback. Usually, bandit feedback refers to the case where the oracle is a \textit{deterministic} zeroth-order oracle and the agent query function is trivial.}
If the agent query function is non-trivial, then we say the agent requires \textit{full-information feedback}.

An adversary $\op{Adv}$ is a set such that each element $\C{B} \in \op{Adv}$, referred to as a \textit{realized adversary}, is a sequence $(\C{B}_1, \cdots, \C{B}_T)$ of functions where each $\C{B}_t$ maps a tuple $(\x_1, \cdots, \x_t) \in \C{K}^t$ to a tuple $(f_t, \C{Q}_t)$ where $f_t \in \BF{F}$ and $\C{Q}_t$ is a query oracle for $f_t$.
We say an adversary $\op{Adv}$ is \textit{oblivious} if for any realization $\C{B} = (\C{B}_1, \cdots, \C{B}_T)$, all functions $\C{B}_t$ are constant, i.e., they are independent of $(\x_1, \cdots, \x_t)$.
In this case, a realized adversary may be simply represented by a sequence of functions $(f_1, \cdots, f_T) \in \BF{F}^T$ and a sequence of query oracles $(\C{Q}_1, \cdots, \C{Q}_T)$ for these functions.
We say an adversary is a \textit{weakly adaptive} adversary if each function $\C{B}_t$ described above does not depend on $\x_t$ and therefore may be represented as a map defined on $\C{K}^{t-1}$.
In this work we also consider adversaries that are \textit{fully adaptive}, i.e., adversaries with no restriction.
Clearly any oblivious adversary is a weakly adaptive adversary and any weakly adaptive adversary is a fully adaptive adversary.
Given a function class $\BF{F}$ and $i \in \{0, 1\}$, we use $\op{Adv}^\t{f}_i(\BF{F})$ to denote the set of all possible realized adversaries with deterministic $i$-th order oracles.
If the oracle is instead stochastic and bounded by $B$, we use $\op{Adv}^\t{f}_i(\BF{F}, B)$ to denote such an adversary.
Finally, we use $\op{Adv}^\t{o}_i(\BF{F})$ and $\op{Adv}^\t{o}_i(\BF{F}, B)$ to denote all oblivious realized adversaries with $i$-th order deterministic and stochastic oracles, respectively.

In order to handle different notions of regret with the same approach, for an agent $\C{A}$, adversary $\op{Adv}$, compact set $\C{U} \subseteq \C{K}^T$, approximation coefficient $0 < \alpha \leq 1$ and $1 \leq a \leq b \leq T$, we define \textit{regret} as
\begin{align*}
	\C{R}_{\alpha, \op{Adv}}^{\C{A}}(\C{U})[a, b] := 
	\sup_{\C{B} \in \op{Adv}} \B{E} \left[ \alpha \max_{\uu = (\uu_1, \cdots, \uu_T) \in \C{U}} \sum_{t = a}^b f_t(\uu_t) - \sum_{t = a}^b f_t(\x_t) \right],
\end{align*}
where the expectation in the definition of the regret is over the randomness of the algorithm and the query oracle.
We use the notation $\C{R}_{\alpha, \C{B}}^{\C{A}}(\C{U})[a, b] := \C{R}_{\alpha, \op{Adv}}^{\C{A}}(\C{U})[a, b]$ when $\op{Adv} = \{\C{B}\}$ is a singleton.
We may drop $\alpha$ when it is equal to 1.
When $\alpha < 1$, we often assume that the functions are non-negative.

\textit{Static adversarial regret} or simply \textit{adversarial regret} corresponds to $a = 1$, $b = T$ and $\C{U} = \C{K}_{\star}^T := \{(\x, \cdots, \x) \mid \x \in \C{K}\}$.
When $a = 1$, $b = T$ and $\C{U}$ contains only a single element then it is referred to as the \textit{dynamic regret} \cite{zinkevich03_onlin,zhang18_adapt_onlin_learn_dynam_envir}. 
\textit{Adaptive regret}, is defined as
$\max_{1 \leq a \leq b \leq T} \C{R}_{\alpha, \op{Adv}}^{\C{A}}(\C{K}_{\star}^T)[a, b]$ \cite{hazan09_effic}. 
We drop $a$, $b$ and $\C{U}$ when the statement is independent of their value or their value is clear from the context.

\section{Proof of Theorem~\ref{thm:main}}\label{app:main}

The proof is similar to the proof of Theorems~2 and~5 in~\cite{pedramfar24_unified_framew_analy_meta_onlin_convex_optim}.

\begin{proof}~

\textbf{Deterministic oracle:}

We first consider the case where $\C{G}$ is a deterministic query oracle for $\F{g}$.
Let $\oo_t = \F{g}(f_t, \x_t)$ denote the output of $\C{G}$ at time-step $t$.
For any realization $\C{B} = (\C{B}_1, \cdots, \C{B}_T) \in \op{Adv}_1^{\t{f}}(\BF{F})$, we define $\C{B}'_t(\x_1, \cdots, \x_t)$ to be the tuple $(q_t, \nabla)$ where
\[
\C{B}'_t(\x_1, \cdots, \x_t) 
:= q_t 
:= \y \mapsto \bra \oo_t, \y - \x_t \ket - \frac{\mu}{2} \| \y - \x_t \|^2,
\]
and $\C{B}' = (\C{B}'_1, \cdots, \C{B}'_T)$.
Note that each $\C{B}'_t$ is a deterministic function of $\x_1, \cdots, \x_t$ and therefore $\C{B}' \in \op{Adv}_1^{\t{f}}(\BF{F}_{\mu, \F{g}})$.
Since the algorithm uses semi-bandit feedback, the sequence of random vectors $(\x_1, \cdots, \x_T)$ chosen by $\C{A}$ is identical between the game with $\C{B}$ and $\C{B}'$.
Therefore, according to definition of quadratizable functions, for any $\y \in \C{K}$, we have
\begin{align*}
\beta \left( q_t(\y) - q_t(\x_t) \right)
= \beta \left( \bra \oo_t, \y - \x_t \ket - \frac{\mu}{2} \| \y - \x_t \|^2 \right)
\geq \alpha f_t(\y) - f_t(h(\x_t)),
\end{align*}
Therefore, we have
\begin{align}\label{eq:main:1}
\max_{\uu \in \C{U}} \left( \alpha \sum_{t = a}^b  f_t(\uu_t) - \sum_{t = a}^b f_t(h(\x_t)) \right)
\leq \beta \max_{\uu \in \C{U}} \left( \sum_{t = a}^b q_t(\uu_t) - \sum_{t = a}^b q_t(\x_t) \right),
\end{align}
Hence
\begin{align*}
\C{R}_{\alpha, \op{Adv}_1^{\t{f}}(\BF{F})}^{\C{A}'}
&= \sup_{\C{B} \in \op{Adv}_1^{\t{f}}(\BF{F})} \C{R}_{\alpha, \C{B}}^{\C{A}'} \\
&= \sup_{\C{B} \in \op{Adv}_1^{\t{f}}(\BF{F})} \B{E} \left[ \max_{\uu \in \C{U}} \left( \alpha \sum_{t = a}^b f_t(\uu_t) - \sum_{t = a}^b  f_t(h(\x_t)) \right) \right] \\
&\leq \sup_{\C{B} \in \op{Adv}_1^{\t{f}}(\BF{F})} \B{E} \left[ \beta \max_{\uu \in \C{U}} \left( \sum_{t = a}^b q_t(\uu_t) - \sum_{t = a}^b q_t(\x_t) \right) \right] \\
&\leq \beta \sup_{\C{B}' \in \op{Adv}_1^{\t{f}}(\BF{F}_{\mu, \F{g}})} \C{R}_{1, \C{B}'}^{\C{A}} 
= \beta \C{R}_{1, \op{Adv}_1^{\t{f}}(\BF{F}_{\mu, \F{g}})}^{\C{A}}.
\end{align*}

\textbf{Stochastic oracle:}

Next we consider the case where $\C{G}$ is a stochastic query oracle for $\F{g}$.

Let $\Omega^\C{Q} = \Omega^\C{Q}_1 \times \cdots \times \Omega^\C{Q}_T$ capture all sources of randomness in the query oracles of $\op{Adv}_1^\t{o}(\BF{F}, B_1)$, i.e., for any choice of $\theta \in \Omega^\C{Q}$, the query oracle is deterministic.
Hence for any $\theta \in \Omega^\C{Q}$ and realized adversary $\C{B} \in \op{Adv}_1^\t{o}(\BF{F}, B_1)$, we may consider $\C{B}_\theta$ as an object similar to an adversary with a deterministic oracle.
However, note that $\C{B}_\theta$ does not satisfy the unbiasedness condition of the oracle, i.e., the returned value of the oracle is not necessarily the gradient of the function at that point.
Recall that $\C{B}_t$ maps a tuple $(\x_1, \cdots, \x_t)$ to a tuple of $f_t$ and a stochastic query oracle for $f_t$.
We will use $\B{E}_{\Omega^\C{Q}}$ to denote the expectation with respect to the randomness of query oracle and $\B{E}_{\Omega^\C{Q}_t}[\cdot] := \B{E}_{\Omega^\C{Q}}[\cdot | f_t, \x_t]$ to denote the expectation conditioned the action of the agent and the adversary.
Similarly, let $\B{E}_{\Omega^\C{A}}$ denote the expectation with respect to the randomness of the agent.
Let $\oo_t$ be the random variable denoting the output of $\C{G}$ at time-step $t$ and let 
\[
\bar{\oo}_t := \B{E}[\oo_t \mid f_t, \x_t] = \B{E}_{\Omega^\C{Q}_t}[\oo_t] = \F{g}(f_t, \x_t).
\]

Similar to the deterministic case, for any realization $\C{B} = (f_1, \cdots, f_T) \in \op{Adv}^\t{o}(\BF{F})$ and any $\theta \in \Omega^\C{Q}$, we define $\C{B}'_{\theta, t}(\x_1, \cdots, \x_t) $ to be the pair $(q_t, \nabla)$ where
\[
q_t 
:= \y \mapsto \bra \oo_t, \y - \x_t \ket - \frac{\mu}{2} \| \y - \x_t \|^2.
\]
We also define $\C{B}'_{\theta} := (\C{B}'_{\theta, 1}, \cdots, \C{B}'_{\theta, T})$.
Note that a specific choice of $\theta$ is necessary to make sure that the function returned by $\C{B}'_{\theta, t}$ is a deterministic function of $\x_1, \cdots, \x_t$ and not a random variable and therefore $\C{B}'_{\theta}$ belongs to $\op{Adv}_1^{\t{f}}(\BF{Q}_\mu[B_1])$.

Since the algorithm uses (semi-)bandit feedback, given a specific value of $\theta$, the sequence of random vectors $(\x_1, \cdots, \x_T)$ chosen by $\C{A}$ is identical between the game with $\C{B}_\theta$ and $\C{B}'_\theta$.
Therefore, for any $\uu_t \in \C{K}$, we have
\begin{align*}
\alpha f_t(\uu_t) - f_t(h(\x_t))
&\leq \beta \left( \bra \bar{\oo}_t, \uu_t - \x_t \ket - \frac{\mu}{2} \| \uu_t - \x_t \|^2 \right) \\
&= \beta \left( \bra \B{E}\left[ \oo_t \mid f_t, \x_t \right], \uu_t - \x_t \ket - \frac{\mu}{2} \| \uu_t - \x_t \|^2 \right) \\
&= \beta \left( \B{E}\left[ \bra \oo_t, \uu_t - \x_t \ket - \frac{\mu}{2} \| \uu_t - \x_t \|^2 \mid f_t, \x_t \right] \right) \\
&= \beta \left( \B{E}\left[ q_t(\uu_t) - q_t(\x_t) \mid f_t, \x_t \right] \right),
\end{align*}
where the first inequality follows from the fact that $f_t$ is up-quadratizable and $\bar{\oo}_t = \F{g}(f_t, \x_t)$.
Therefore we have
\begin{align*}
\B{E} \left[
\alpha \sum_{t = a}^b f_t(\uu_t) - \sum_{t = a}^b  f_t(h(\x_t))
\right]
&\leq
\beta \B{E} \left[
\sum_{t = a}^b \B{E} \left[ q_t(\uu_t) - q_t(\x_t) | f_t, \x_t \right]
\right] \\
&=
\beta \B{E} \left[
\sum_{t = a}^b q_t(\uu_t) - q_t(\x_t)
\right].
\end{align*}
Since $\C{B}$ is oblivious, the sequence $(f_1, \cdots, f_T)$ is not affected by the randomness of query oracles or the agent.
Therefore we have
\begin{align*}
\C{R}_{\alpha, \C{B}}^{\C{A}}
&= 
\B{E} \left[
\alpha \max_{\uu \in \C{U}} \sum_{t = a}^b f_t(\uu_t) - \sum_{t = a}^b  f_t(h(\x_t))
\right] \\
&= 
\max_{\uu \in \C{U}}
\B{E} \left[
\alpha \sum_{t = a}^b f_t(\uu_t) - \sum_{t = a}^b  f_t(h(\x_t))
\right] \\
&\leq
\beta \max_{\uu \in \C{U}}
\B{E} \left[
\sum_{t = a}^b q_t(\uu_t) - \sum_{t = a}^b q_t(\x_t)
\right] \\
&\leq
\beta \B{E} \left[
\max_{\uu \in \C{U}}
\left(
\sum_{t = a}^b q_t(\uu_t) - \sum_{t = a}^b q_t(\x_t)
\right)
\right] 
= \beta \B{E} \left[
\C{R}_{1, \C{B}'_\theta}^{\C{A}}
\right],
\end{align*}
where the second inequality follows from Jensen's inequality.
Hence we have
\begin{align*}
\C{R}_{\alpha, \op{Adv}_1^\t{o}(\BF{F}, B_1)}^{\C{A}}
&= 
\sup_{\C{B} \in \op{Adv}_1^\t{o}(\BF{F}, B_1)}
\C{R}_{\alpha, \C{B}}^{\C{A}}
\leq 
\sup_{\C{B} \in \op{Adv}_1^\t{o}(\BF{F}, B_1), \theta \in \Omega^\C{Q}}
\beta \C{R}_{1, \C{B}'_\theta}^{\C{A}} \\
&\leq 
\sup_{\C{B}' \in \op{Adv}_1^{\t{f}}(\BF{Q}_\mu[B_1])}
\beta \C{R}_{1, \C{B}'}^{\C{A}}
= 
\beta \C{R}_{1, \op{Adv}_1^{\t{f}}(\BF{Q}_\mu[B_1])}^{\C{A}}
\qedhere
\end{align*}
\end{proof}

\section{Proof of Lemma~\ref{lem:dr_mono_general:quad}}\label{app:dr_mono_general:quad}

\begin{proof}
We have $\x \vee \y + \x \wedge \y = \x  + \y$.
Therefore, following the definition of curvature, we have
\begin{align*}
f(\x \vee \y) - f(\y)
\geq (1 - c) (f(\x) - f(\x \wedge \y)).
\end{align*}
Since $f$ is non-negative, this implies that
\begin{equation}
\begin{aligned}
\label{eq:lem:dr_mono_general:quad:1} 
(f(\x \vee \y) - f(\x))
&+ \frac{1}{\gamma^2} (f(\x \wedge \y) - f(\x)) \\
&= (f(\y) - f(\x)) 
+ (f(\x \vee \y) - f(\y)) 
+ \frac{1}{\gamma^2} (f(\x \wedge \y) - f(\x)) \\
&\geq (f(\y) - f(\x)) + (1 - c - \frac{1}{\gamma^2}) (f(\x) - f(\x \wedge \y)) \\
&= f(\y) - (c + \frac{1}{\gamma^2}) f(\x) + (- 1 + c + \frac{1}{\gamma^2}) f(\x \wedge \y) \\
&\geq f(\y) - (c + \frac{1}{\gamma^2}) f(\x).    
\end{aligned}
\end{equation}
On the other hand, applying the up-concavity inequalities with respect to the specified field $\tilde{\nabla} f$, we have
\begin{align*}
f(\x \vee \y) - f(\x)
&\leq \frac{1}{\gamma} \left( \bra \tilde{\nabla} f(\x), \x \vee \y - \x \ket - \frac{\mu}{2} \| \x \vee \y - \x \|^2 \right), \\
f(\x) - f(\x \wedge \y)
&\geq \gamma \left( \bra \tilde{\nabla} f(\x), \x - \x \wedge \y \ket + \frac{\mu}{2} \| \x - \x \wedge \y \|^2 \right).
\end{align*}
Therefore, using Inequality~\ref{eq:lem:dr_mono_general:quad:1} and the fact that $f(\x \wedge \y) \geq 0$, we see that
\begin{align*}
f(\y) - \frac{1 + c\gamma^2}{\gamma^2} f(\x)
&\leq (f(\x \vee \y) - f(\x))
+ \frac{1}{\gamma^2} (f(\x \wedge \y) - f(\x)) \\
&\leq \frac{1}{\gamma} \left( \bra \tilde{\nabla} f(\x), \x \vee \y - \x \ket 
- \frac{\mu}{2} \| \x \vee \y - \x \|^2
+ \bra \tilde{\nabla} f(\x), \x \wedge \y - \x \ket \right. \\
&\qquad\qquad \left. - \frac{\mu}{2} \| \x - \x \wedge \y \|^2 \right) \\
&= \frac{1}{\gamma} \left( \bra \tilde{\nabla} f(\x), \x \vee \y + \x \wedge \y - 2\x \ket
- \frac{\mu}{2} \| \x \vee \y - \x \|^2
- \frac{\mu}{2} \| \x - \x \wedge \y \|^2 \right) \\
&= \frac{1}{\gamma} \left( \bra \tilde{\nabla} f(\x), \y - \x \ket
- \frac{\mu}{2} \| \x - \y \|^2 \right),
\end{align*}
where we used $\x \vee \y + \x \wedge \y = \x + \y$ and
\begin{align*}
\| \x \vee \y - \x \|^2 + \| \x - \x \wedge \y \|^2
&= \sum_{[y]_i \geq [x]_i} (y_i - x_i)^2
+ \sum_{[y]_i < [x]_i} (x_i - y_i)^2 \\
&= \sum_{i} (x_i - y_i)^2
= \| \x - \y \|^2
\end{align*}
in the last equality.
The claim now follows from multiply both sides by $\frac{\gamma^2}{1 + c\gamma^2}$.
\end{proof}

\section{Proof of Lemma~\ref{lem:dr_mono_zero:quad}}\label{app:dr_mono_zero:quad}

\begin{proof}
Clearly we have $F(\BF{0}) = 0$.
For any $\x \neq \BF{0}$, the integrand in the definition of $F$ is a continuous non-negative function of $z$ that is bounded by
\begin{align*}
\frac{\gamma e^{\gamma (z - 1)}}{(1 - e^{-\gamma})z} (f(z * \x) - f(\BF{0})) 
\leq \frac{\gamma e^{\gamma (z - 1)}}{(1 - e^{-\gamma})z} M_1 \| z * \x \|
\leq \frac{\gamma}{1 - e^{-\gamma}} M_1 \| \x \|.
\end{align*}
Therefore $F$ is well-defined on $[0, 1]^d$.
Moreover, we have
\footnote{For this gradient identity, everywhere differentiability can be relaxed: it suffices that $f$ is Lipschitz on $[0,1]^d$, is differentiable at the anchor $\BF{0}$, and has a full gradient at almost every point, with respect to length, of every nondegenerate line segment joining $\BF{0}$ to a point in the cube. In particular, it suffices that $f$ is Lipschitz and its nondifferentiability set has zero one-dimensional Hausdorff measure and does not contain $\BF{0}$. This relaxation concerns the gradient identity only; the remaining assumptions of the lemma are unchanged.}
\begin{align*}
\int_0^1 \frac{\gamma e^{\gamma (z - 1)}}{(1 - e^{-\gamma})} \nabla f(z * \x) dz
= \nabla \int_0^1 \frac{\gamma e^{\gamma (z - 1)}}{(1 - e^{-\gamma})z} (f(z * \x) - f(\BF{0})) dz
= \nabla F(\x).
\end{align*}
It follows that $F$ is differentiable everywhere and $\B{E}\left[ \nabla f(\C{Z} * \x) \right] = \nabla F(\x)$.
To prove the last claim, first we note that
\begin{align*}
\frac{1 - e^{-\gamma}}{\gamma} \bra \nabla F(\x), \x \ket 
&= \left\bra \int_0^1 e^{\gamma (z - 1)} \nabla f(z * \x) dz, \x \right\ket \\
&= \int_0^1 e^{\gamma (z - 1)} \bra \nabla f(z * \x), \x \ket dz \\
&= \int_0^1 e^{\gamma (z - 1)} d f(z * \x) \\
&= e^{\gamma (z - 1)} f(z * \x)|_{z = 0}^{z = 1} - \int_0^1 f(z * \x) \frac{d e^{\gamma (z - 1)}}{dz} dz \\
&= f(\x) - e^{-\gamma} f(\BF{0}) - \int_0^1 \gamma e^{\gamma (z - 1)} f(z * \x) dz.
\end{align*}
On the other hand, using monotonicity and up-concavity of $f$, we have
\begin{align*}
\frac{1 - e^{-\gamma}}{\gamma} \bra \nabla F(\x), \y \ket 
&= \int_0^1 e^{\gamma (z - 1)} \bra \nabla f(z * \x), \y \ket dz \\
&\geq \int_0^1 e^{\gamma (z - 1)} \bra \nabla f(z * \x), \y \vee (z * \x) - z * \x \ket dz \\
&\geq \int_0^1 \gamma e^{\gamma (z - 1)} \left( f(\y \vee (z * \x)) - f(z * \x) \right) dz \\
&\geq \int_0^1 \gamma e^{\gamma (z - 1)} \left( f(\y) - f(z * \x) \right) dz \\
&= (1 - e^{-\gamma}) f(\y)
- \left( \int_0^1 \gamma e^{\gamma (z - 1)} f(z * \x) dz \right), \\
\end{align*}
where we used $\int_0^1 e^{\gamma (z - 1)} \gamma dz = 1 - e^{-\gamma}$ in the last equality.
Therefore
\begin{align*}
\frac{1 - e^{-\gamma}}{\gamma} \bra \nabla F(\x), \y - \x \ket 
&\geq (1 - e^{-\gamma}) f(\y) - f(\x) + e^{-\gamma} f(\BF{0})
\geq (1 - e^{-\gamma}) f(\y) - f(\x).
\qedhere
\end{align*}
\end{proof}

%

\section{Proof of Lemma~\ref{lem:dr_nonmono:quad}} \label{app:dr_nonmono:quad}

We start by extending a useful lemma from the literature.
Versions of the following lemma for DR-submodular functions appeared in~\cite{bian17_contin_dr_maxim,chekuri15_multip_weigh_updat_concav_submod_funct_maxim}. 
Here we use the extension of Lemma~2.2 in~\cite{mualem22_resol_approx_offlin_onlin_non} to up-concave functions.
The proof is included for completeness.

\begin{lemma}\label{lem:f_join_non-monotone}
For any two vectors $\x, \y \in [0, 1]^d$ and any continuously differentiable non-negative up-concave function $f$ we have
\[
f(\x \vee \y) \geq (1 - \|\x\|_\infty) f(\y).
\]
\end{lemma}

\begin{proof}[Proof of Lemma~\ref{lem:f_join_non-monotone}]
If $\norm{\x}_\infty = 0$, then $\x$ is the zero vector, and the lemma is trivially true.
On the other hand, if $\x \vee \y = \y$, the lemma follows from non-negativity of $f$.
Thus, we may assume that $\z := \x \vee \y - \y \geq \BF{0}$, $\z \neq \BF{0}$, and $\norm{\x}_\infty > 0$.
We have
\begin{align} 
f(\x \vee \y) - f(\y)
&=
\int_0^1 \left. \frac{df(\y + r \cdot \z)}{dr}\right|_{r = t} dt
=
\int_0^1 \langle \z, \nabla f(\y + t \cdot \z)\rangle dt \nonumber \\
&=
\|\x\|_\infty \cdot \int_0^{1/\|\x\|_\infty} \langle \z, \nabla f(\y + \|\x\|_\infty \cdot t' \cdot \z)\rangle dt' \nonumber \\
&\geq
\|\x\|_\infty \cdot \int_0^{1/\|\x\|_\infty} \langle \z, \nabla f(\y + t' \cdot \z)\rangle dt', \label{eq:diff_z}
\end{align}
where \eqref{eq:diff_z} 
holds by changing the integration variable to $t' = t / \|\x\|_\infty$, and the inequality follows from up-concavity of $f$, in particular because $f$ is concave along the line segment $[\y, \y + t'.\z] \subseteq [0, 1]^d$.
To see that the last inclusion holds, note that, for every $1 \leq i \leq d$, if $x_i \leq y_i$, then $y_i + t' \cdot z_i = y_i \leq 1$, and if $x_i \geq y_i$, then
\begin{align*}
y_i + t' \cdot z_i
\leq
y_i + \frac{z_i}{\|\x\|_\infty}
=
y_i + \frac{x_i - y_i}{\|\x\|_\infty}
\leq
\frac{x_i}{\|\x\|_\infty}
\leq 1.
\end{align*}

Next we see that
\begin{align*}
\int_0^{1/\|\x\|_\infty} \langle \z, \nabla f(\y + t' \cdot \z)\rangle dt'
&= \int_0^{1/\|\x\|_\infty} \langle \z, \nabla f(\y + t' \cdot \z)\rangle dt' \\ 
&= \int_0^{1/\|\x\|_\infty} \left. \frac{df(\y + r \cdot \z)}{dr}\right|_{r = t'} dt' \\
&= f\left(\y + \frac{\z}{\|\x\|_\infty}\right) - f(\y)
\geq - f(\y),
\end{align*}
where the inequality follows from non-negativity of $f$.
The lemma now follows by plugging this inequality into Inequality~\eqref{eq:diff_z} and rearranging the terms.
\end{proof}

\begin{proof}[Proof of Lemma~\ref{lem:dr_nonmono:quad}]

Clearly we have $F(\underline{\x}) = 0$.
For any $\x \neq \underline{\x}$, the integrand in the definition of $F$ is a continuous function of $z$ that is bounded by
\begin{align*}
\left| \frac{2}{3 z (1 - \frac{z}{2})^3} \left( f\left(\frac{z}{2} * (\x - \underline{\x}) + \underline{\x} \right) - f(\underline{\x}) \right) \right|
&\leq \frac{2}{3 z (1 - \frac{z}{2})^3} M_1 \| \frac{z}{2} * (\x - \underline{\x}) \|
\leq \frac{8}{3} M_1 \| \x - \underline{\x} \|.
\end{align*}
Therefore $F$ is well-defined on $[0, 1]^d$.
Moreover, we have 
\footnote{For this gradient identity, everywhere differentiability can be relaxed: it suffices that $f$ is Lipschitz on $[0,1]^d$, is differentiable at the anchor $\underline{\x}$, and has a full gradient at almost every point, with respect to length, of every nondegenerate line segment joining $\underline{\x}$ to a point in the cube. In particular, it suffices that $f$ is Lipschitz and its nondifferentiability set has zero one-dimensional Hausdorff measure and does not contain $\underline{\x}$. This relaxation concerns the gradient identity only; the remaining assumptions of the lemma are unchanged.}
\begin{align*}
\int_0^1 \frac{1}{3 (1 - \frac{z}{2})^3} \nabla f\left(\frac{z}{2} * (\x - \underline{\x}) + \underline{\x} \right) dz 
&= \nabla \int_0^1 \frac{2}{3 z (1 - \frac{z}{2})^3} \left( f\left(\frac{z}{2} * (\x - \underline{\x}) + \underline{\x} \right) - f(\underline{\x}) \right) dz \\
&= \nabla F(\x) 
\end{align*}
It follows that $F$ is differentiable everywhere and $\B{E}\left[ \nabla f\left(\frac{\C{Z}}{2} * (\x - \underline{\x}) + \underline{\x} \right) \right] = \nabla F(\x)$.

To prove the last claim, let $\x_z := \frac{z}{2} * (\x - \underline{\x}) + \underline{\x}$ and $\omega(z) = \frac{1}{8 (1 - \frac{z}{2})^3}$.
We have
\begin{align*}
\frac{3}{8} \bra \nabla F(\x), \y \ket 
&= \int_0^1 \omega(z) \bra \nabla f(\x_z), \y \ket dz \\
&= \int_0^1 \omega(z) \left( 
    \bra \nabla f(\x_z), \y - \x_z \wedge \y \ket 
    + \bra \nabla f(\x_z), \x_z \wedge \y - \x_z \ket 
    + \bra \nabla f(\x_z), \x_z \ket \right) dz \\
&= \int_0^1 \omega(z) \left( 
    \bra \nabla f(\x_z), \x_z \vee \y - \x_z \ket 
    + \bra \nabla f(\x_z), \x_z \wedge \y - \x_z \ket 
    + \bra \nabla f(\x_z), \x_z \ket \right) dz \\
&\geq \int_0^1 \omega(z) \left( 
    ( f(\x_z \vee \y) - f(\x_z) )
    + ( f(\x_z \wedge \y) - f(\x_z) )
    + \bra \nabla f(\x_z), \x_z \ket \right) dz \\
&\geq \int_0^1 \omega(z) \left( 
    f(\x_z \vee \y) - 2 f(\x_z)
    + \bra \nabla f(\x_z), \x_z \ket \right) dz.
\end{align*}
Using Lemma~\ref{lem:f_join_non-monotone}, we have
\begin{align*}
f(\x_z \vee \y)
&\geq (1 - \|\x_z\|_\infty) f(y) \\
&\geq \left(1 - \left(\left(1 - \frac{z}{2}\right) \|\underline{\x}\|_\infty + \frac{z}{2} \|\x\|_\infty \right)\right) f(y) \\
&\geq \left(1 - \left(\left(1 - \frac{z}{2}\right) \|\underline{\x}\|_\infty + \frac{z}{2}\right)\right) f(y) \\
&= \left(1 - \frac{z}{2}\right) \left(1 - \|\underline{\x}\|_\infty \right) f(y).
\end{align*}
Therefore
\begin{align}\label{eq:dr_nonmono:quad:1}
\frac{3}{8} \bra \nabla F(\x), \y \ket 
&\geq \int_0^1 \omega(z) \left( 
    \left(1 - \frac{z}{2}\right) \left(1 - \|\underline{\x}\|_\infty \right) f(y)
    - 2 f(\x_z)
    + \bra \nabla f(\x_z), \x_z \ket \right) dz.
\end{align}

Next we bound $\bra \nabla F(\x), \x \ket$.
We have
\begin{align*}
\frac{3}{8} \bra \nabla F(\x), \x \ket
= \int_0^1 \omega(z) \bra \nabla f(\x_z), \x - \x_z \ket dz
    + \int_0^1 \omega(z) \bra \nabla f(\x_z), \x_z \ket dz
\end{align*}
For the first term, we have
\begin{align*}
\int_0^1 \omega(z) \bra \nabla f(\x_z), \x - \x_z \ket dz
&= \int_0^1 \omega(z) \left\bra \nabla f(\x_z), \left(1 - \frac{z}{2}\right)(\x - \underline{\x}) \right\ket dz \\
&= \int_0^1 (2 - z) \omega(z) \left\bra \nabla f(\x_z), \frac{\x - \underline{\x}}{2} \right\ket dz \\
&= \int_0^1 (2 - z) \omega(z) df(\x_z) \\
&= (2 - z) \omega(z) f(\x_z) |_{z = 0}^1
    - \int_0^1 ((2 - z) \omega'(z) - \omega(z)) f(\x_z) dz \\
&= f(\x_1) - \frac{1}{4} f(\underline{\x})
    - \int_0^1 \frac{1}{4 (1 - \frac{z}{2})^3} f(\x_z) dz,
\end{align*}
which implies that
\begin{align}\label{eq:dr_nonmono:quad:2}
\frac{3}{8} \bra \nabla F(\x), \x \ket
&= f(\x_1) - \frac{1}{4} f(\underline{\x})
    - \int_0^1 \frac{1}{4 (1 - \frac{z}{2})^3} f(\x_z) dz
    + \int_0^1 \omega(z) \bra \nabla f(\x_z), \x_z \ket dz \nonumber \\
&\leq f(\x_1) 
    - \int_0^1 2\omega(z) f(\x_z) dz
    + \int_0^1 \omega(z) \bra \nabla f(\x_z), \x_z \ket dz
\end{align}

Combining Equations~\ref{eq:dr_nonmono:quad:1} and~\ref{eq:dr_nonmono:quad:2}, we have
\begin{align*}
\frac{3}{8} \bra \nabla F(\x), \y - \x \ket 
&\geq \int_0^1 \omega(z) \left( 
    \left(1 - \frac{z}{2}\right) \left(1 - \|\underline{\x}\|_\infty \right) f(y)
    - 2 f(\x_z)
    + \bra \nabla f(\x_z), \x_z \ket \right) dz \\
&\qquad- f(\x_1) 
    + \int_0^1 2\omega(z) f(\x_z) dz
    - \int_0^1 \omega(z) \bra \nabla f(\x_z), \x_z \ket dz \\
&= \frac{1 - \|\underline{\x}\|_\infty}{4} f(\y) \int_0^1 4 \left(1 - \frac{z}{2}\right) \omega(z) dz 
    - f(\x_1) \\
&= \frac{1 - \|\underline{\x}\|_\infty}{4} f(\y) - f\left(\frac{\x + \underline{\x}}{2} \right).
\qedhere
\end{align*}
\end{proof}

\section{Proof of Theorem~\ref{thm:first-order-to-zero-order}}\label{app:first-order-to-zero-order}

The algorithms are special cases of Algorithms~2 and~3 in~\cite{pedramfar24_unified_framew_analy_meta_onlin_convex_optim} where the shrinking parameter and the smoothing parameter are equal.
The only difference is that we use rescaling so that the base algorithm $\C{A}$ is run on the same convex set $\C{K}$.

\begin{algorithm2e}[H]
\SetKwInOut{Input}{Input}\DontPrintSemicolon
\caption{First order to zeroth order - $\mathtt{FOTZO}(\C{A})$}
\label{alg:first-order-to-zeroth-order}
\small
\Input{Domain $\C{K}$, Linear space $\C{L}_0$, Reference point $\BF{c} \in \C{K}$ and $r > 0$ such that $\op{aff}(\C{K}) \cap \B{B}_r(\BF{c}) \subseteq \C{K}$, smoothing parameter $0 < \delta < r$, horizon $T$, algorithm $\C{A}$}
$k \gets \op{dim}(\C{L}_0)$ \;
\For{$t = 1, 2, \dots, T$}{
$\x_t \gets $ the action chosen by $\C{A}$ \;
Play $\x_t$ \;
Let $f_t$ be the function chosen by the adversary \;
\For{$i$ starting from 1, while $\C{A}^\t{query}$ is not terminated for this time-step}{
    Sample $\vv_{t, i} \in \B{S}^1 \cap \C{L}_0$ uniformly \;
    Let $\y_{t, i}$ be the query chosen by $\C{A}^\t{query}$ \;
    Query the oracle at the point $T_\delta(\y_{t, i}) + \delta \vv_{t, i}$ to get $o_{t, i}$ \;
    Pass $(1 - \frac{\delta}{r})\frac{k}{\delta} o_{t, i} \vv_{t, i}$ as the oracle output to $\C{A}$ \;
}
}
\end{algorithm2e}

\begin{algorithm2e}[H]
\SetKwInOut{Input}{Input}\DontPrintSemicolon
\caption{Semi-bandit to bandit - $\mathtt{STB}(\C{A})$}
\label{alg:semi-bandit-to-bandit}
\small
\Input{Domain $\C{K}$, Linear space $\C{L}_0$, Reference point $\BF{c} \in \C{K}$ and $r > 0$ such that $\op{aff}(\C{K}) \cap \B{B}_r(\BF{c}) \subseteq \C{K}$, smoothing parameter $0 < \delta < r$, horizon $T$, algorithm $\C{A}$}
$k \gets \op{dim}(\C{L}_0)$ \;
\For{$t = 1, 2, \dots, T$}{
Sample $\vv_t \in \B{S}^1 \cap \C{L}_0$ uniformly \;
$\x_t \gets $ the action chosen by $\C{A}$ \;
Play $T_\delta(\x_t) + \delta \vv_t$ \;
Let $f_t$ be the function chosen by the adversary \;
Let $o_t$ be the output of the value oracle \;
Pass $(1 - \frac{\delta}{r}) \frac{k}{\delta} o_t \vv_t$ as the oracle output to $\C{A}$ \;
}
\end{algorithm2e}

\begin{algorithm2e}[H]
\SetKwInOut{Input}{Input}\DontPrintSemicolon
\caption{First order to zeroth order with Two Point gradient estimator - $\mathtt{FOTZO\textrm{-}2P}(\C{A})$}
\label{alg:first-order-to-det-zeroth-order}
\small
\Input{Domain $\C{K}$, Linear space $\C{L}_0$, Reference point $\BF{c} \in \C{K}$ and $r > 0$ such that $\op{aff}(\C{K}) \cap \B{B}_r(\BF{c}) \subseteq \C{K}$, smoothing parameter $0 < \delta < r$, horizon $T$, algorithm $\C{A}$}
$k \gets \op{dim}(\C{L}_0)$ \;
\For{$t = 1, 2, \dots, T$}{
$\x_t \gets $ the action chosen by $\C{A}$ \;
Play $\x_t$ \;
Let $f_t$ be the function chosen by the adversary \;
\For{$i$ starting from 1, while $\C{A}^\t{query}$ is not terminated for this time-step}{
    Sample $\vv_{t, i} \in \B{S}^1 \cap \C{L}_0$ uniformly \;
    Let $\y_{t, i}$ be the query chosen by $\C{A}^\t{query}$ \;
    Query the deterministic oracle at the points $T_\delta(\y_{t, i}) + \delta \vv_{t, i}$ and $T_\delta(\y_{t, i}) - \delta \vv_{t, i}$ \;
    Pass $(1 - \frac{\delta}{r})\frac{k}{2 \delta} \left( f_t(T_\delta(\y_{t, i}) + \delta \vv_{t, i}) - f_t(T_\delta(\y_{t, i}) - \delta \vv_{t, i}) \right) \vv_{t, i}$ as the oracle output to $\C{A}$ \;
}
}
\end{algorithm2e}

\begin{proof}[Proof of Theorems~\ref{thm:first-order-to-zero-order} and~\ref{thm:first-order-to-det-zero-order}]
Fix an oblivious realized adversary $\C{B}$, represented by functions $(f_t)_{t=1}^T$ and their value oracles $(\C{Q}_t)_{t=1}^T$.
Write $s=1-\delta/r$ and let $\vv$ be uniform on $\B{S}^1\cap\C{L}_0$.
For $\mathtt{FOTZO}$ and $\mathtt{STB}$, the simulated first-order oracle is
\begin{align}\label{eq:semi-bandit-to-bandit:1}
\C{Q}'_t(\x)=s\frac{k}{\delta}\C{Q}_t(T_\delta(\x)+\delta\vv)\vv.
\end{align}
The ball-smoothing identity gives
\[
\B{E}[\C{Q}'_t(\x)]
=s\nabla\hat f_t(T_\delta(\x))=\nabla f'_t(\x),
\qquad
\|\C{Q}'_t(\x)\|\leq\frac{kB_0}{\delta}.
\]
Gradients here are taken along $\op{aff}(\C{K})$.
The identity follows by identifying $\C{L}_0$ with $\B{R}^k$ and applying the usual ball-smoothing identity; see also Remark~4 in~\cite{pedramfar23_unified_approac_maxim_contin_dr_funct}.

For $\mathtt{FOTZO\textrm{-}2P}$, the value oracle is deterministic and the simulated oracle is
\begin{align}\label{eq:first-order-to-det-zero-order:1}
\C{Q}'_t(\x)
=s\frac{k}{2\delta}\bigl(f_t(T_\delta(\x)+\delta\vv)-f_t(T_\delta(\x)-\delta\vv)\bigr)\vv.
\end{align}
Symmetry of $\vv$ and the same smoothing identity imply $\B{E}[\C{Q}'_t(\x)]=\nabla f'_t(\x)$.
Moreover, Lipschitz continuity gives
\[
\|\C{Q}'_t(\x)\|
\leq s\frac{k}{2\delta}M_1\|(T_\delta(\x)+\delta\vv)-(T_\delta(\x)-\delta\vv)\|
\leq kM_1.
\]
Both queried points belong to $\C{K}$ by the definition of $T_\delta$.
Thus the simulated adversary $\C{B}'$ belongs to $\op{Adv}_1^\t{o}(\BF{F}'_\delta,kB_0/\delta)$ in the one-point case and to $\op{Adv}_1^\t{o}(\BF{F}'_\delta,kM_1)$ in the two-point case.
Coupling the simulated oracle responses identifies the base algorithm's actions with a sequence $(\x_t)_{t=1}^T$ in each wrapper.

For every $\x\in\C{K}$, smoothing and rescaling give the uniform estimate
\begin{align}\label{eq:smoothing-rescaling-error}
|f'_t(\x)-f_t(\x)|
&\leq |\hat f_t(T_\delta(\x))-f_t(T_\delta(\x))|
      +|f_t(T_\delta(\x))-f_t(\x)| \nonumber\\
&\leq \delta M_1+\frac{\delta}{r}M_1\|\BF{c}-\x\|
\leq (1+D/r)\delta M_1.
\end{align}
For $\mathtt{FOTZO}$ and $\mathtt{FOTZO\textrm{-}2P}$, the played action is $\x_t$, so this also bounds the difference between the base and wrapper rewards.
For $\mathtt{STB}$, the played action is $\BF{p}_t=T_\delta(\x_t)+\delta\vv_t$, and instead
\[
|f_t(\BF{p}_t)-f'_t(\x_t)|
\leq |f_t(\BF{p}_t)-f_t(T_\delta(\x_t))|
     +|f_t(T_\delta(\x_t))-\hat f_t(T_\delta(\x_t))|
\leq 2\delta M_1.
\]
Let $\epsilon_{\t{comp}}=(1+D/r)\delta M_1$ and let $\epsilon_{\t{act}}$ be $(1+D/r)\delta M_1$ for either full-information wrapper and $2\delta M_1$ for $\mathtt{STB}$.
For any compact comparator set $\C{U}\subseteq\C{K}^T$ and interval $[a,b]$, the uniform comparator estimate implies
\[
\max_{\uu\in\C{U}}\sum_{t=a}^b f_t(\uu_t)
\leq\max_{\uu\in\C{U}}\sum_{t=a}^b f'_t(\uu_t)+(b-a+1)\epsilon_{\t{comp}}.
\]
Combining this with the reward estimate yields
\begin{align*}
\C{R}_{\alpha,\C{B}}^{\C{A}'}(\C{U})[a,b]
&\leq\C{R}_{\alpha,\C{B}'}^{\C{A}}(\C{U})[a,b]
 +(b-a+1)(\alpha\epsilon_{\t{comp}}+\epsilon_{\t{act}})\\
&\leq\C{R}_{\alpha,\C{B}'}^{\C{A}}(\C{U})[a,b]
 +(3+2D/r)\delta M_1T.
\end{align*}
Taking the supremum over $\C{B}$ proves both theorems.
\end{proof}

\begin{corollary}\label{cor:first-order-to-zero-order}
Under the assumptions of Theorem~\ref{thm:first-order-to-zero-order}, suppose that, uniformly for the smoothing parameters used and oracle bounds $B_1$,
$\C{R}_{\alpha,\op{Adv}_1^\t{o}(\BF{F}'_\delta,B_1)}^{\C{A}}=O(B_1T^\eta)$ for some $0\leq\eta<1$.
Choose $\delta=\min\{r/2,T^{(\eta-1)/2}\}$ and let $\C{A}'=\mathtt{FOTZO}(\C{A})$, or $\C{A}'=\mathtt{STB}(\C{A})$ when $\C{A}$ is semi-bandit.
Then
\[
\C{R}_{\alpha,\op{Adv}_0^\t{o}(\BF{F},B_0)}^{\C{A}'}
=O((kB_0+M_1)T^{(1+\eta)/2}).
\]
For fixed geometry and dimension, and $M_1=O(B_0)$, this is $O(B_0T^{(1+\eta)/2})$.
\end{corollary}

\begin{corollary}\label{cor:first-order-to-det-zero-order}
Under the assumptions of Theorem~\ref{thm:first-order-to-det-zero-order}, choose $\delta=\min\{r/2,T^{-1}\}$ and $\C{A}'=\mathtt{FOTZO\textrm{-}2P}(\C{A})$.
Then
\[
\C{R}_{\alpha,\op{Adv}_0^\t{o}(\BF{F})}^{\C{A}'}
\leq\C{R}_{\alpha,\op{Adv}_1^\t{o}(\BF{F}'_\delta,kM_1)}^{\C{A}}
+O(M_1).
\]
Thus any uniform bound for the base algorithm over $\BF{F}'_\delta$ transfers with $B_1=kM_1$ and an additive $O(M_1)$ term.
\end{corollary}

\section{Proof of Theorem~\ref{thm:stoch-fi-to-sb}}\label{app:stoch-fi-to-sb}

\begin{proof}
Fix $\C{B}\in\op{Adv}_i^\t{o}(\BF{F},B)\{T\}$, with functions $(f_t)_{t=1}^T$ and query oracles $(\C{Q}_t)_{t=1}^T$.
For each block $q$, define
\[
\hat f_q=\frac1L\sum_{t=(q-1)L+1}^{qL}f_t.
\]
Closure under convex combinations ensures $\hat f_q\in\BF{F}$.
Conditioned on the history before the block and the selected query points, the random permutation assigns each query a uniformly random time in the block.
Thus the response to each fixed query is an unbiased $i$-th order observation of $\hat f_q$, bounded by $B$.
The joint oracle responses are coupled using the same permutation as in $\mathtt{SFTT}$; queries do not depend on other responses within the block.
This defines a realized block adversary $\hat{\C{B}}\in\op{Adv}_i^\t{o}(\BF{F},B)\{T/L\}$ whose interaction with $\C{A}$ has the same law as the simulated interaction.

Let $J$ be the union of the complete blocks contained in $[a,b]$, and let $E=[a,b]\setminus J$ be the remaining boundary rounds.
There are at most $2(L-1)$ rounds in $E$.
For every fixed comparator $\uu_0\in\C{K}$, Lipschitz continuity and, when $\alpha<1$, nonnegativity imply
\[
\alpha f_t(\uu_0)-f_t(\x_t)
\leq f_t(\uu_0)-f_t(\x_t)\leq M_1D.
\]
Consequently, removing the boundary rounds incurs at most $2M_1D(L-1)$ in regret.
If $J$ is empty, this already proves the claim.
Otherwise, the complete blocks are precisely $q=a',\ldots,b'$.
Writing $\hat\x_q$ for the base action, we have
\begin{align*}
\C{R}_{\alpha,\C{B}}^{\C{A}'}(\C{K}_\star^T)[a,b]
&\leq 2M_1D(L-1)
 +\B{E}\left[\alpha\max_{\uu_0\in\C{K}}\sum_{t\in J}f_t(\uu_0)-\sum_{t\in J}f_t(\x_t)\right]\\
&=2M_1D(L-1)
 +\B{E}\left[\sum_{q=a'}^{b'}\sum_{t=(q-1)L+1}^{qL}
       (f_t(\hat\x_q)-f_t(\x_t))\right]\\
&\quad+L\B{E}\left[\alpha\max_{\uu_0\in\C{K}}\sum_{q=a'}^{b'}\hat f_q(\uu_0)
                  -\sum_{q=a'}^{b'}\hat f_q(\hat\x_q)\right]\\
&\leq 2M_1D(L-1)+M_1DKN_{a,b}
 +L\C{R}_{\alpha,\hat{\C{B}}}^{\C{A}}(\C{K}_\star^{T/L})[a',b'].
\end{align*}
The last inequality holds because the played action differs from $\hat\x_q$ on at most $K$ rounds per block.
Taking the supremum over $\C{B}$ proves the theorem for both feedback orders.
\end{proof}

\begin{remark}
The boundedness of oracle outputs is used only to ensure that the block oracle belongs to the stated adversary class.
The same argument applies to other oracle classes closed under the block mixtures.
If $L$ does not divide $T$, run the construction on the first $L\lfloor T/L\rfloor$ rounds and play any fixed feasible action on the remaining rounds; this adds at most $M_1D(L-1)$ to an interval regret bound.
\end{remark}

\begin{corollary}\label{cor:stoch-fi-to-sb}
Under the assumptions of Theorem~\ref{thm:stoch-fi-to-sb}, suppose the base algorithm $\C{A}$ has a uniform interval regret bound $O(BS^\eta)$ at horizon $S$, where $0\leq\eta<1$.
Suppose $K=O(T^\theta)$ for some $0\leq\theta<1$ and choose an integer
$L=\Theta(T^{(1+\theta-\eta)/(2-\eta)})$ with $L>K$.
Then, with the final incomplete block handled as above,
\[
\C{R}_{\alpha,\op{Adv}_i^\t{o}(\BF{F},B)}^{\C{A}'}(\C{K}_\star^T)[a,b]
=O\left((B+M_1D)T^{\frac{(1+\theta)(1-\eta)+\eta}{2-\eta}}\right).
\]
For $K=O(1)$, the exponent is $1/(2-\eta)$.
When $M_1D=O(B)$, the coefficient simplifies to $O(B)$.
\end{corollary}
\begin{proof}
Theorem~\ref{thm:stoch-fi-to-sb} and the boundary-round bound give
\begin{align*}
\C{R}_{\alpha,\op{Adv}_i^\t{o}(\BF{F},B)}^{\C{A}'}(\C{K}_\star^T)[a,b]
&=O(M_1DKT/L)+O(LB(T/L)^\eta)+O(M_1DL).
\end{align*}
The stated choice of $L$ balances the exponents of the first two terms.
The exponent of the last term is no larger, since the difference is $\eta(1-\theta)/(2-\eta)\geq0$.
\end{proof}

\section{Proof of Theorem~\ref{thm:online-to-offline}}\label{app:online-to-offline}

\begin{proof}
Since $\C{A}$ requires $T^{\theta}$ queries per time-step, it requires a total of $T^{1 + \theta}$ queries.
The expected error is bounded by the regret divided by time.
Hence we have $\epsilon = O(T^{\eta - 1})$ after $T^{1 + \theta}$ queries.
Therefore, the total number of queries to keep the error bounded by $\epsilon$ is $O(\epsilon^{-\frac{1 + \theta}{1 - \eta}})$.
\end{proof}

\section{Projection-free adaptive regret}\label{sec:w-ada-reg}\label{app:w-ada-reg}

The $\mathtt{SO\textrm{-}OGD}$ algorithm in~\cite{garber22_new_projec_algor_onlin_convex} is a deterministic algorithm with semi-bandit feedback, designed for online convex optimization with a deterministic gradient oracle.
Here \textit{we assume that the separation oracle is deterministic.}

Here we use the notation $\BF{c}$, $r$ and $\hat{\C{K}}_\delta$ described in Section~\ref{sec:meta}.

\begin{algorithm2e}[H]
\SetKwInOut{Input}{Input}\DontPrintSemicolon\LinesNumbered
\caption{Online Gradient Ascent via a Separation Oracle - $\mathtt{SO\textrm{-}OGA}$}
\label{alg:SO-OGA}
\small
\Input{ horizon $T$, constraint set $\C{K}$, step size $\eta$ }
$\x_1 \gets \BF{c} \in \hat{\C{K}}_\delta$ \; 
\For{$t = 1, 2, \dots, T$}{
Play $\x_t$ and observe $\oo_t = \nabla f_t(\x_t)$ \;
$\x'_{t+1} = \x_t + \eta \oo_t$ \;
Set $\x_{t+1} = \mathtt{SO\textrm{-}IP}_{\C{K}}(\x'_{t+1})$, the output of Algorithm~\ref{alg:SO-IP} with initial point $\x'_{t+1}$ \;
}
\end{algorithm2e}

Note that here we use a maximization version of the algorithm, which we denote by $\mathtt{SO\textrm{-}OGA}$.
Here $\BF{P}_{\C{K}}$ denotes projection into the convex set $\C{K}$.
The original version, which is designed for minimization, uses the update rule $\x_{t+1}' = \x_t - \eta \oo_t$ in Algorithm~\ref{alg:SO-OGA} instead.

\begin{algorithm2e}[H]
\SetKwInOut{Input}{Input}\DontPrintSemicolon\LinesNumbered
\caption{Infeasible Projection via a Separation Oracle - $\mathtt{SO\textrm{-}IP}_{\C{K}}(\y_0)$}
\label{alg:SO-IP}
\small
\Input{ Constraint set $\C{K}$, shrinking parameter $0 < \delta < r$, initial point $\y_0$ }
$\y'_1 \gets \BF{P}_{\op{aff}(\C{K})}(\y_0)$ \;
$\y_1 \gets \BF{c} + \frac{\y'_1 - \BF{c}}{\max\{1, \|\y'_1 - \BF{c}\|/D \}}$
\tcc*{$\y_1$ is projection of $\y_0$ over $\B{B}_{D}(\BF{c}) \cap \op{aff}(\C{K})$}
\For{$i = 1, 2, \dots$}{
Call $\op{SO}_{\C{K}}$ with input $\y_i$ \;
\eIf{$\y_i \notin \C{K}$}{
Set $\g_i$ to be the hyperplane returned by $\op{SO}_{\C{K}}$ 
\tcc*{$\forall \x \in \C{K}, \quad \bra \y_i - \x, \g_i \ket > 0$}
$\g'_i \gets \BF{P}_{\op{aff}(\C{K}) - \BF{c}}(\g_i)$ \;
Update $\y_{i+1} \gets \y_i - \delta \frac{\g'_i}{\|\g'_i\|} $ \;
}{
Return $\y \gets \y_i$ \;
}
}
\end{algorithm2e}

\begin{lemma}\label{lem:w-ada-reg:infeasible-projection}
Algorithm~\ref{alg:SO-IP} stops after at most $(\op{dist}(\y_0, \hat{\C{K}}_\delta)^2 - \op{dist}(\y, \hat{\C{K}}_\delta)^2)/\delta^{2} + 1$ iterations and returns $\y \in \C{K}$ such that $\forall \z \in \hat{\C{K}}_\delta$, we have $\| \y - \z \| \leq \| \y_0 - \z \|$.
\end{lemma}
\begin{proof}
We first note that this algorithm is invariant under translations.
Hence it is sufficient to prove the result when $\BF{c} = 0$.

Let $\op{SO}'_\C{K}$ denote the following separation oracle.
If $\y \in \C{K}$ or $\y \notin \op{aff}(\C{K})$, then $\op{SO}'_\C{K}$ returns the same output as $\op{SO}_\C{K}$.
Otherwise, it returns $\BF{P}_{\op{aff}(\C{K})}(\g)$ where $\g \in \B{R}^d$ is the output of $\op{SO}_\C{K}$.
To prove that this is indeed a separation oracle, we only need to consider the case where $\y \in \op{aff}(\C{K}) \setminus \C{K}$.
We know that $\g$ is a vector such that 
\[
\forall \x \in \C{K}, \quad \bra \y - \x, \g \ket > 0.
\]
Since $\BF{P}_{\op{aff}(\C{K})}$ is an orthogonal projection, we have 
\[
\bra \y - \x, \BF{P}_{\op{aff}(\C{K})}(\g) \ket 
= \bra \BF{P}_{\op{aff}(\C{K})}(\y - \x), \g \ket 
= \bra \y - \x, \g \ket 
> 0.
\]
for all $\x \in \C{K}$, which implies that $\op{SO}'_\C{K}$ is a separation oracle.

Now we see that Algorithm~\ref{alg:SO-IP} is an instance of Algorithm~6 in~\cite{garber22_new_projec_algor_onlin_convex} applied to the initial point $\y_1$ using the separation oracle $\op{SO}'_\C{K}$, with normalized shrinking parameter $\delta/r$ and its additional shrinking parameter set to zero.
Hence we may use Lemma~13 in~\cite{garber22_new_projec_algor_onlin_convex} directly to see that Algorithm~\ref{alg:SO-IP} stops after at most $(\op{dist}(\y_1, \hat{\C{K}}_\delta)^2 - \op{dist}(\y, \hat{\C{K}}_\delta)^2)/\delta^{2} + 1$ iterations and returns $\y \in \C{K}$ such that $\forall \z \in \hat{\C{K}}_\delta$, we have $\| \y - \z \| \leq \| \y_1 - \z \|$.
Since $\BF{c}\in\op{aff}(\C{K})$, the two initialization steps make $\y_1$ the Euclidean projection of $\y_0$ onto the closed convex set $\B{B}_D(\BF{c})\cap\op{aff}(\C{K})$.
This set contains $\C{K}$, because $\BF{c}\in\C{K}$ and $D=\op{diam}(\C{K})$, and hence also contains $\hat{\C{K}}_\delta$.
Projection therefore gives $\|\y_1-\z\|\leq\|\y_0-\z\|$ for every $\z\in\hat{\C{K}}_\delta$, and consequently $\op{dist}(\y_1,\hat{\C{K}}_\delta)\leq\op{dist}(\y_0,\hat{\C{K}}_\delta)$.
It follows that Algorithm~\ref{alg:SO-IP} stops after at most
\begin{align*}
\frac{\op{dist}(\y_1, \hat{\C{K}}_\delta)^2 - \op{dist}(\y, \hat{\C{K}}_\delta)^2}{\delta^{2}} + 1
&\leq \frac{\op{dist}(\y_0, \hat{\C{K}}_\delta)^2 - \op{dist}(\y, \hat{\C{K}}_\delta)^2}{\delta^{2}} + 1
\end{align*}
steps and
\begin{align*}
\forall \z \in \hat{\C{K}}_\delta \subseteq \op{aff}(\C{K})
,\quad
\| \y - \z \| &\leq \| \y_1 - \z \| \leq \| \y_0 - \z \|.
\qedhere
\end{align*}
\end{proof}

In the following, we use the notation 
\[
\C{AR}^{\C{A}}_{\alpha, \op{Adv}} := \max_{1 \leq a \leq b \leq T} \C{R}^{\C{A}}_{\alpha, \op{Adv}} (\C{K}_*^T) [a, b],
\]
to denote the adaptive regret.

\begin{theorem}\label{thm:w-ada-reg:orignial}
Let $\BF{L}$ be a class of linear functions over $\C{K}$ such that $\|l\| \leq M_1$ for all $l \in \BF{L}$ and let $D = \op{diam}(\C{K})$.
Fix $v > 0$ such that $vT^{-1/2}\in(0,1)$, and set $\delta = rvT^{-1/2}\in(0,r)$ and $\eta = \frac{v r}{2 M_1} T^{-1/2}$.
Then we have
\begin{align*}
\C{AR}_{1, \op{Adv}_1^{\t{f}}(\BF{L})}^\mathtt{SO\textrm{-}OGA}
    &= O( M_1 T^{1/2} ).
\end{align*}
\end{theorem}
\begin{proof}
Since the algorithm is deterministic, according to Theorem~1 in~\cite{pedramfar24_unified_framew_analy_meta_onlin_convex_optim}, it is sufficient to prove this regret bound against the oblivious adversary $\op{Adv}_1^\t{o}(\BF{L})$.

Note that this algorithm is invariant under translations.
Hence it is sufficient to prove the result when $\BF{c} = 0$.
If $\op{aff}(\C{K}) = \B{R}^d$, then we have $\B{B}_r(\BF{0}) \subseteq \C{K} \subseteq \B{B}_R(\BF{0})$ and we may use Theorem~14 from~\cite{garber22_new_projec_algor_onlin_convex}, with normalized shrinking parameter $\delta/r=vT^{-1/2}$, to obtain the desired result for the oblivious adversary $\op{Adv}_1^\t{o}(\BF{L})$.
On the other hand, the assumption $\B{B}_r(\BF{0}) \subseteq \C{K}$ is only used in the proof of Lemma~13 in~\cite{garber22_new_projec_algor_onlin_convex}.
Here we use Lemma~\ref{lem:w-ada-reg:infeasible-projection} instead which does not require this assumption.
\end{proof}

The following corollary is an immediate consequence of the above theorem and Theorems~\ref{thm:dr_mono_general:main},~\ref{thm:dr_mono_zero:main},~\ref{thm:dr_nonmono:main},~\ref{thm:online-to-offline} and Corollaries~\ref{cor:first-order-to-zero-order},~\ref{cor:first-order-to-det-zero-order} and~\ref{cor:stoch-fi-to-sb}.

\begin{corollary}\label{cor:w-ada-reg}
Let $\mathtt{SO\textrm{-}OGA}$ denote the algorithm described above.
For the zeroth-order conclusions, assume that each function admits a differentiable Lipschitz extension to $[0,1]^d$ satisfying the up-concavity assumptions of the corresponding theorem with respect to its own gradient, together with that theorem's remaining assumptions, as in Remark~\ref{rem:smoothing-conversion-details}.
The first-order conclusions in part (a) allow extensions that are up-concave with respect to the specified vector field in Theorem~\ref{thm:dr_mono_general:main}, without requiring differentiability.
In part (c), choose the reference point $\underline{\x}\in\argmin_{\z\in\C{K}}\|\z\|_\infty$, so $h=\min_{\z\in\C{K}}\|\z\|_\infty$.
The hidden constants may depend on the fixed geometric parameters of $\C{K}$.
Then the following are true.
\begin{enumerate}[label=\alph*)]
\item Under the assumptions of Theorem~\ref{thm:dr_mono_general:main}, we have
\begin{align*}
\C{AR}_{\frac{\gamma^2}{1+c\gamma^2},\op{Adv}_1^{\t{f}}(\BF{F})}^{\mathtt{SO\textrm{-}OGA}}
&\leq O(M_1T^{1/2}),\\
\C{AR}_{\frac{\gamma^2}{1+c\gamma^2},\op{Adv}_1^{\t{o}}(\BF{F},B_1)}^{\mathtt{SO\textrm{-}OGA}}
&\leq O(B_1T^{1/2}).
\end{align*}
Under the additional extension assumption for zeroth-order feedback,
\[
\C{AR}_{\frac{\gamma^2}{1+c\gamma^2},\op{Adv}_0^{\t{o}}(\BF{F})}^{\mathtt{FOTZO\textrm{-}2P}(\mathtt{SO\textrm{-}OGA})}
\leq O(M_1T^{1/2}).
\]
If $\BF{F}$ is also bounded by $M_0$ and $B_0\geq M_0$, then
\[
\C{AR}_{\frac{\gamma^2}{1+c\gamma^2},\op{Adv}_0^{\t{o}}(\BF{F},B_0)}^{\mathtt{STB}(\mathtt{SO\textrm{-}OGA})}
\leq O((B_0+M_1)T^{3/4}).
\]
\item Under the assumptions of Theorem~\ref{thm:dr_mono_zero:main}, we have:
\begin{align*}
\C{AR}_{1 - e^{-\gamma}, \op{Adv}_1^\t{o}(\BF{F}, B_1)}^{ \C{A} } 
    &\leq O( B_1 T^{1/2} ), \\
\C{AR}_{1 - e^{-\gamma}, \op{Adv}_0^\t{o}(\BF{F})}^{ \mathtt{FOTZO\textrm{-}2P}(\C{A}) }
    &\leq O( M_1 T^{1/2} ),
\end{align*}
where $\C{A} = \mathtt{OMBQ}(\mathtt{SO\textrm{-}OGA}, \mathtt{BQM0}, \op{Id})$.
Note that $\C{A}$ is a first order full-information algorithm that requires a single query per time-step.
If we also assume $\BF{F}$ is bounded by $M_0$ and $B_0 \geq M_0$, then
\begin{align*}
\C{AR}_{1 - e^{-\gamma}, \op{Adv}_1^\t{o}(\BF{F}, B_1)}^{ \C{A}_{\op{semi-bandit}} } 
    &\leq O( B_1 T^{2/3} ), \\
\C{AR}_{1 - e^{-\gamma}, \op{Adv}_0^\t{o}(\BF{F}, B_0)}^{ \C{A}_{\op{full-info--0}} } 
    &\leq O( (B_0 + M_1) T^{3/4} ) \\
\C{AR}_{1 - e^{-\gamma}, \op{Adv}_0^\t{o}(\BF{F}, B_0)}^{ \C{A}_{\op{bandit}} } 
    &\leq O( (B_0 + M_1) T^{4/5} )
\end{align*}
where 
\begin{align*}
\C{A}_{\op{semi-bandit}} &= \mathtt{SFTT}(\C{A}) 
,\quad
\C{A}_{\op{full-info--0}} = \mathtt{FOTZO}(\C{A})
,\quad
\C{A}_{\op{bandit}}   = \mathtt{SFTT}(\C{A}_{\op{full-info--0}}).
\end{align*}

\item Under the assumptions of Theorem~\ref{thm:dr_nonmono:main}, we have:
\begin{align*}
\C{AR}_{\frac{1 - h}{4}, \op{Adv}_1^\t{o}(\BF{F}, B_1)}^{\C{A}}
  &\leq O( B_1 T^{1/2} ), \\
\C{AR}_{\frac{1 - h}{4}, \op{Adv}_0^\t{o}(\BF{F})}^{\mathtt{FOTZO\textrm{-}2P}(\C{A})}
    &\leq O( M_1 T^{1/2} ),
\end{align*}
where $\C{A} = \mathtt{OMBQ}(\mathtt{SO\textrm{-}OGA}, \mathtt{BQN}, \x \mapsto \frac{\x + \underline{\x}}{2})$.
Note that $\C{A}$ is a first order full-information algorithm that requires a single query per time-step.
If we also assume $\BF{F}$ is bounded by $M_0$ and $B_0 \geq M_0$, then
\begin{align*}
\C{AR}_{\frac{1 - h}{4}, \op{Adv}_1^\t{o}(\BF{F}, B_1)}^{ \C{A}_{\op{semi-bandit}} } 
    &\leq O( B_1 d^{1/2} T^{2/3} ),\\
\C{AR}_{\frac{1 - h}{4}, \op{Adv}_0^\t{o}(\BF{F}, B_0)}^{\C{A}_{\op{full-info--0}}}
    &\leq O( (B_0 + M_1) d^{1/2} T^{3/4} ),\\
\C{AR}_{\frac{1 - h}{4}, \op{Adv}_0^\t{o}(\BF{F}, B_0)}^{ \C{A}_{\op{bandit}} } 
    &\leq O( (B_0 + M_1) d^{1/2} T^{4/5} ),
\end{align*}
where
\begin{align*}
\C{A}_{\op{semi-bandit}} &= \mathtt{SFTT}(\C{A}) 
,\quad
\C{A}_{\op{full-info--0}} = \mathtt{FOTZO}(\C{A})
,\quad
\C{A}_{\op{bandit}}   = \mathtt{SFTT}(\C{A}_{\op{full-info--0}}).
\end{align*}
\end{enumerate}
\end{corollary}

\section{Dynamic regret}\label{app:d-reg}

Improved Ader ($\mathtt{IA}$) algorithm~\cite{zhang18_adapt_onlin_learn_dynam_envir} is a deterministic algorithm with semi-bandit feedback, designed for online convex optimization with a deterministic gradient oracle.

\begin{algorithm2e}[H]
\SetKwInOut{Input}{Input}\DontPrintSemicolon
\caption{Improved Ader - $\mathtt{IA}$}
\label{alg:IA}
\small
\Input{ horizon $T$, constraint set $\C{K}$, step size $\lambda$, a set $\C{H}$ containing step sizes for experts }
Activate a set of experts $\{ E^\eta \mid \eta \in \C{H} \}$ by invoking Algorithm~\ref{alg:IA:expert} for each step size $\eta \in \C{H}$ \;
Sort step sizes in ascending order $\eta_1 \leq \cdots \leq \eta_N$, and set $w_1^{\eta_i} = \frac{C}{i(i+1)}$ where $C = 1 + \frac{1}{|\C{H}|}$\;
\For{$t = 1, 2, \dots, T$}{
Receive $\x_t^\eta$ from each expert $E^\eta$ \;
Play the action $\x_t = \sum_{\eta \in \C{H}} w_t^\eta \x_t^\eta$ and observe $\oo_t = \nabla f_t(\x_t)$ \;
Define $l_t(\y) := \bra \oo_t, \y - \x_t \ket$ \;
Update the weight of each expert by
$w_{t+1}^\eta = \frac{w_t^\eta e^{\lambda l_t(\x_t^\eta)}}{\sum_{\mu \in \C{H}} w_t^\mu e^{\lambda l_t(\x_t^\mu)}}$.\;
Send the gradient $\oo_t$ to each expert $E^\eta$ \;
}
\end{algorithm2e}

\begin{algorithm2e}[H]
\SetKwInOut{Input}{Input}\DontPrintSemicolon
\caption{Improved Ader : Expert algorithm}
\label{alg:IA:expert}
\small
\Input{ horizon $T$, constraint set $\C{K}$, step size $\eta$ }
Let $\x_1^\eta$ be any point in $\C{K}$ \;
\For{$t = 1, 2, \dots, T$}{
Send $\x_t^\eta$ to the main algorithm \;
Receive $\oo_t$ from the main algorithm \;
$\x_{t+1}^\eta = \BF{P}_{\C{K}}(\x_t^\eta + \eta \oo_t)$ \;
}
\end{algorithm2e}
In Algorithm~\ref{alg:IA:expert}, $\BF{P}_{\C{K}}$ denotes projection into the convex set $\C{K}$.
Note that here we used the maximization version of this algorithm.
The original version, which is designed for minimization, uses the update rule $\x_{t+1}^\eta = \BF{P}_{\C{K}}(\x_t^\eta - \eta \oo_t)$ in Algorithm~\ref{alg:IA:expert} and negative exponents in the expert-weight update instead.

\begin{theorem}\label{thm:orignial-d-reg}
Let $\BF{L}$ be a class of linear functions over $\C{K}$ such that $\|l\| \leq M_1$ for all $l \in \BF{L}$ and let $D = \op{diam}(\C{K})$.
Set $\C{H} := \{ \eta_i = \frac{2^{i-1} D}{M_1} \sqrt{\frac{7}{2T}}  \mid 1 \leq i \leq N \}$ where $N = \lceil \frac{1}{2} \log_2 (1 + 4T/7) \rceil + 1$ and $\lambda = \sqrt{2/(T M_1^2 D^2)}$.
Then for any comparator sequence $\uu \in \C{K}^T$, we have
\begin{align*}
\C{R}_{1, \op{Adv}_1^{\t{f}}(\BF{L})}^\mathtt{IA}(\uu)
    &= O( M_1 \sqrt{T(1 + P_T(\uu))} ).
\end{align*}
\end{theorem}
\begin{proof}
If we use the oblivious adversary $\op{Adv}_1^\t{o}(\BF{L})$ instead, this theorem is simply a restatement of the special case (i.e. when the functions are linear) of Theorem~4 in~\cite{zhang18_adapt_onlin_learn_dynam_envir}.
\footnote{We note that although Theorem~4 in~\cite{zhang18_adapt_onlin_learn_dynam_envir} assumes that the convex set contains the origin, this assumption is not really needed.
In fact, for any arbitrary convex set, we may first translate it to contain the origin, apply Theorem~4 and then translate it back to obtain the results for the original convex set.
}
Since the algorithm is deterministic, according to Theorem~1 in~\cite{pedramfar24_unified_framew_analy_meta_onlin_convex_optim}, the regret bound remains unchanged when we replace $\op{Adv}_1^\t{o}(\BF{L})$ with $\op{Adv}_1^{\t{f}}(\BF{L})$.
\end{proof}

The following corollary is an immediate consequence of the above theorem and Theorems~\ref{thm:dr_mono_general:main},~\ref{thm:dr_mono_zero:main},~\ref{thm:dr_nonmono:main}, and Corollaries~\ref{cor:first-order-to-zero-order} and~\ref{cor:first-order-to-det-zero-order}.

Note that we do not use the meta-algorithm $\mathtt{OTB}$ since Improved Ader is designed for non-stationary regret and does not offer any advantages in the offline case.
On the other hand, we do not use the meta-algorithm $\mathtt{SFTT}$ in this case since Theorem~\ref{thm:stoch-fi-to-sb} is only for the setting where the comparator is $\C{K}_*^T$ and does not allow us to convert bounds for dynamic regret.

\begin{corollary}\label{cor:d-reg}
Let $\mathtt{IA}$ denote ``Improved Ader'' described above.
For the zeroth-order conclusions, assume that each function admits a differentiable Lipschitz extension to $[0,1]^d$ satisfying the up-concavity assumptions of the corresponding theorem with respect to its own gradient, together with that theorem's remaining assumptions, as in Remark~\ref{rem:smoothing-conversion-details}.
The first-order conclusions in part (a) allow extensions that are up-concave with respect to the specified vector field in Theorem~\ref{thm:dr_mono_general:main}, without requiring differentiability.
In part (c), choose the reference point $\underline{\x}\in\argmin_{\z\in\C{K}}\|\z\|_\infty$, so $h=\min_{\z\in\C{K}}\|\z\|_\infty$.
The hidden constants may depend on the fixed geometric parameters of $\C{K}$.
Then the following are true.
\begin{enumerate}[label=\alph*)]
\item Under the assumptions of Theorem~\ref{thm:dr_mono_general:main}, we have
\begin{align*}
\C{R}_{\frac{\gamma^2}{1+c\gamma^2},\op{Adv}_1^{\t{f}}(\BF{F})}^{\mathtt{IA}}(\uu)
&=O(M_1\sqrt{T(1+P_T(\uu))}),\\
\C{R}_{\frac{\gamma^2}{1+c\gamma^2},\op{Adv}_1^{\t{o}}(\BF{F},B_1)}^{\mathtt{IA}}(\uu)
&=O(B_1\sqrt{T(1+P_T(\uu))}).
\end{align*}
Under the additional extension assumption for zeroth-order feedback,
\[
\C{R}_{\frac{\gamma^2}{1+c\gamma^2},\op{Adv}_0^{\t{o}}(\BF{F})}^{\mathtt{FOTZO\textrm{-}2P}(\mathtt{IA})}(\uu)
=O(M_1\sqrt{T(1+P_T(\uu))}).
\]
If $\BF{F}$ is also bounded by $M_0$ and $B_0\geq M_0$, then
\[
\C{R}_{\frac{\gamma^2}{1+c\gamma^2},\op{Adv}_0^{\t{o}}(\BF{F},B_0)}^{\mathtt{STB}(\mathtt{IA})}(\uu)
=O((B_0+M_1)T^{3/4}(1+P_T(\uu))^{1/2}).
\]
\item Under the assumptions of Theorem~\ref{thm:dr_mono_zero:main}, we have:
\begin{align*}
\C{R}_{1 - e^{-\gamma}, \op{Adv}_1^\t{o}(\BF{F}, B_1)}^{\C{A}}(\uu)
    &= O( B_1 \sqrt{T(1 + P_T(\uu))} ) \\
\C{R}_{1 - e^{-\gamma}, \op{Adv}_0^\t{o}(\BF{F})}^{\mathtt{FOTZO\textrm{-}2P}(\C{A})}(\uu)
    &= O( M_1 \sqrt{T(1 + P_T(\uu))} ).
\end{align*}
where $\C{A} = \mathtt{OMBQ}(\mathtt{IA}, \mathtt{BQM0}, \op{Id})$.
Note that $\C{A}$ is a first order full-information algorithm that requires a single query per time-step.
If we also assume $\BF{F}$ is bounded by $M_0$ and $B_0 \geq M_0$, then
\begin{align*}
\C{R}_{1 - e^{-\gamma}, \op{Adv}_0^\t{o}(\BF{F}, B_0)}^{ \C{A}_{\op{full-info--0}} } 
(\uu)
    &= O( (B_0 + M_1) T^{3/4} (1 + P_T(\uu))^{1/2} ).
\end{align*}
where $\C{A}_{\op{full-info--0}} = \mathtt{FOTZO}(\C{A})$.

\item Under the assumptions of Theorem~\ref{thm:dr_nonmono:main}, we have:
\begin{align*}
\C{R}_{\frac{1 - h}{4}, \op{Adv}_1^\t{o}(\BF{F}, B_1)}^{\C{A}}(\uu)
    &= O( B_1 \sqrt{T(1 + P_T(\uu))} ), \\
\C{R}_{\frac{1 - h}{4}, \op{Adv}_0^\t{o}(\BF{F})}^{\mathtt{FOTZO\textrm{-}2P}(\C{A})}(\uu)
    &= O( M_1 \sqrt{T(1 + P_T(\uu))} ).
\end{align*}
where $\C{A} = \mathtt{OMBQ}(\mathtt{IA}, \mathtt{BQN}, \x \mapsto \frac{\x + \underline{\x}}{2})$.
Note that $\C{A}$ is a first order full-information algorithm that requires a single query per time-step.
If we also assume $\BF{F}$ is bounded by $M_0$ and $B_0 \geq M_0$, then
\begin{align*}
\C{R}_{\frac{1 - h}{4}, \op{Adv}_0^\t{o}(\BF{F}, B_0)}^{\C{A}_{\op{full-info--0}}}(\uu)
    &= O( (B_0 + M_1) T^{3/4} (1 + P_T(\uu))^{1/2} ).
\end{align*}
where $\C{A}_{\op{full-info--0}} = \mathtt{FOTZO}(\C{A})$.
\end{enumerate}
\end{corollary}

\end{document}